\documentclass[reqno,final]{amsart}

\usepackage[utf8]{inputenc} 
\usepackage[english]{babel}
\usepackage{mathabx}
\usepackage{amstext}
\usepackage{euscript}
\usepackage{bbm}
\usepackage{mathrsfs}
\usepackage{enumerate}
\usepackage{graphicx}
\usepackage{caption}
\usepackage{amscd}
\usepackage{amssymb}
\usepackage{amsthm}
\usepackage{amsmath}
\usepackage{amstext}
\usepackage{amsfonts}
\usepackage{latexsym}
\usepackage{mathtools} 
\usepackage{enumitem} 
\usepackage{accents} 
\usepackage[foot]{amsaddr} 

\usepackage{verbatim} 

\usepackage{mathtools} 

\usepackage{enumitem} 

\usepackage[english]{babel} 

\usepackage{esint}

\usepackage{calrsfs}

\usepackage[usenames,dvipsnames]{xcolor} 
\definecolor{dkblue}{RGB}{1,31,91} 
\usepackage[colorlinks=true, pdfstartview=FitV, linkcolor=dkblue, citecolor=dkblue, urlcolor=dkblue]{hyperref}

\newcommand{\Blue}{}

\usepackage{cite}

\theoremstyle{definition}
\newtheorem{theorem}{Theorem}
\newtheorem{corollary}[theorem]{Corollary}
\newtheorem{lemma}[theorem]{Lemma}
\newtheorem{definition}[theorem]{Definition}
\newtheorem{proposition}[theorem]{Proposition}
\newtheorem{remark}[theorem]{Remark}

\newtheorem{hypothesis}[theorem]{Hypothesis}

\numberwithin{equation}{section}
\numberwithin{theorem}{section}

\newcommand{\R}{\mathbb{R}}

\newcommand{\eqdef }{\overset{\mbox{\tiny{def}}}{=}}
\newcommand{\rtw}{{\mathbb{R}^2}}

\begin{document}

\keywords{Vlasov equation, Maxwell equations, Plasma dynamics, and Magnetic confinement.}
\subjclass[2010]{Primary 35Q83, 35Q61, 82C40, and 82D10. }

\title[Magnetic confinement for the Vlasov-Maxwell system]{Magnetic confinement for the 2D axisymmetric relativistic Vlasov-Maxwell system in an annulus}

\author[J. W. Jang]{Jin Woo Jang$^*$}
\address{$^*$Department of Mathematics, Pohang University of Science and Technology, Pohang, Republic of Korea.  \href{mailto:jangjw@postech.ac.kr}{jangjw@postech.ac.kr} (\href{https://orcid.org/0000-0002-3846-1983}{https://orcid.org/0000-0002-3846-1983})}
\thanks{$^*$Supported by the German DFG grant CRC 1060, the Korean Basic Science Research Institute Fund NRF-2021R1A6A1A10042944 and the Korean IBS grant IBS-R003-D1.}

\author[R. M. Strain]{Robert M. Strain$^\dagger$}
\address{$^\dagger$Department of Mathematics, University of Pennsylvania, Philadelphia, PA 19104, USA.  \href{mailto:strain@math.upenn.edu}{strain@math.upenn.edu} (\href{https://orcid.org/0000-0002-1107-8570}{https://orcid.org/0000-0002-1107-8570})}
\thanks{$^\dagger$Partially supported by the NSF grants DMS-1764177 and DMS-2055271 of the USA}

\author[T. K. Wong]{Tak Kwong Wong$^\ddagger$}
\address{$^\ddagger$Department of mathematics, The University of Hong Kong, Pokfulam, Hong Kong. \href{mailto:takkwong@maths.hku.hk}{takkwong@maths.hku.hk} (\href{https://orcid.org/0000-0002-6454-984X}{https://orcid.org/0000-0002-6454-984X})}
\thanks{$^\ddagger$Partially supported by the HKU Seed Fund for Basic Research under the project code 201702159009, the Start-up Allowance for Croucher Award Recipients, Hong Kong General Research Fund (GRF) grant ``Solving Generic Mean Field Type Problems: Interplay between Partial Differential Equations and Stochastic Analysis'' with project number 17306420, and Hong Kong GRF grant ``Controlling the Growth of Classical Solutions of a Class of Parabolic Differential Equations with Singular Coefficients: Resolutions for Some Lasting Problems from Economics’’ with project number 17302521.}

\begin{abstract}
  Although the nuclear fusion process has received a great deal of attention in recent years, the amount of mathematical analysis that supports the stability of the system seems to be relatively insufficient. This paper deals with the mathematical analysis of the magnetic confinement of the plasma via kinetic equations. We prove the global wellposedness of the \textit{Vlasov-Maxwell} system in a two-dimensional annulus when a huge (\textit{but finite-in-time}) external magnetic potential is imposed near the boundary. We assume that the solution is axisymmetric. The authors hope that this work is a step towards a more generalized work on the three-dimensional Tokamak structure. The highlight of this work is the physical assumptions on the external magnetic potential well which remains finite \textit{within a finite time interval} and from that, we prove that the plasma never touches the boundary. In addition, we provide a sufficient condition on the magnitude of the external magnetic potential to guarantee that the plasma is confined in an annulus of the desired thickness which is slightly larger than the initial support. Our method uses the cylindrical coordinate forms of the \textit{Vlasov-Maxwell} system.  
\end{abstract}

\dedicatory{In memory of Robert Glassey}

\thispagestyle{empty}
\maketitle
\tableofcontents

\section{Introduction}

\subsection{Motivation}
This paper studies the effect of a large external magnetic field on the initial and boundary value problem for the two-dimensional relativistic \textit{Vlasov-Maxwell} system for initial data of unrestricted size. 

The magnetic confinement of a collisionless plasma has received a great deal of attention in both mathematical and numerical perspectives, as it is the main principle of the nuclear fusion process and there is no doubt that nuclear fusion is one possible future production of electrical energy \cite{ongena2016magnetic}. The dynamics of a plasma have been interpreted numerically and analytically via the magnetohydrodynamic (MHD) fluid equation and the kinetic \textit{Vlasov-Maxwell} system, though the computational challenges for the simulations of the tokamak process have been studied in \cite{fasoli2016computational}. This paper deals with the kinetic approach to the magnetic confinement of the plasma via a theoretical study of the relativistic \textit{Vlasov-Maxwell} system.

 Indeed, the \textit{Vlasov} equation without the presence of an external magnetic field has been extensively studied. Here we introduce a small number of results from the literature on the \textit{Vlasov-Maxwell} system by Degond \cite{MR870991}, DiPerna-Lions \cite{MR1003433}, Glassey-Strauss \cite{MR816621, MR881551}, Glassey-Schaeffer \cite{MR969207, MR1066384, MR1463042, MR1620506, MR1777635} , Horst \cite{MR1088081}, Guo \cite{MR1224079, MR1354697, MR1459591, MR1671928},  Rein \cite{MR1086751}, Bouchut-Golse-Pallard \cite{MR2012645}, Klainerman-Staffilani \cite{MR3318184}, and Strain \cite{MR2259206}. Recent results include the work on the Strichartz estimates \cite{MR3437855} and continuation criteria \cite{MR816621, MR3373411,MR3248056,MR3721415}. Regarding the rigorous derivation of the Vlasov equations, see the work of Dobrushin \cite{dobrushin1979vlasov}.

Regarding the situation where one applies a large external magnetic field to the system, the general theory of confining devices such as tokamaks and stellarators were studied in \cite{MR1266242, MR1850780}. 
Regarding the magnetic confinement for the \textit{Vlasov-Poisson} system, we mention the work of \cite{MR2733252, MR2997588, MR3148082, MR3318184, MR3723427} and the numerical results of \cite{MR3709891,MR3568167,MR3485969,MR3782409}. 
Regarding the magnetic confinement problem for the \textit{Vlasov-Maxwell} system in the presence of the effect of the self-consistent magnetic field, we mention the analytic proof by Nguyen-Nguyen-Strauss in a two-dimensional infinite strip with a symmetry in $x_2$-direction \cite{MR3294216,MR3404774}. Also, we introduce that Filbet and Rodrigues in \cite{1811.09087} generalized the work of \cite{MR3568167} in the large magnetic field limit.

To the best of authors' knowledge, there has been no result on the magnetic confinement of the full \textit{Vlasov-Maxwell} system by a finite external magnetic potential, with which the plasma \textit{never} collides with the boundary. The authors believe that, even though one can show that the plasma will eventually converge to some confined steady states in time, the confinement will not be ideal if the plasma can collide with the boundary during some initial time-interval. This is because just a tiny amount of hot particles will destroy the nuclear fusion reactor in reality.

This paper is devoted to introducing an analytic proof for the magnetic confinement for the two-dimensional \textit{Vlasov-Maxwell} system in an annulus by a large \textit{but finite-in-time} external magnetic potential, which is the first step to the full three dimensional analysis in a toroidal geometry for the actual tokamaks or stellarators. We prove that the finite potential is large enough to confine the plasma for all time in $[0,T]$ such that the plasma never touches the boundary. In addition, we provide a sufficient condition on the magnitude of the external magnetic potential to guarantee that the plasma is confined to an annulus of the desired thickness which is slightly larger than the initial support.

One of the difficulties that arises in the implementation of the magnetic confinement by imposing a finite external magnetic potential is the fact that the self-consistent electromagnetic fields are also growing in time and that these fields indeed affect the behavior of the particle trajectory. 
In order to implement it, one first needs to estimate the precise upper-bounds of the growth of the self-consistent fields that are coupled to the \textit{Vlasov} equation. In the geometry of an annulus, another difficulty arises because the fields can be accelerating each other via the \textit{Maxwell} equations even under the symmetry in the angular direction. This can be shown via estimating the fields as solutions to some wave-type equations by the method of characteristics and via obtaining an energy identity that is related to the Poynting theorem in cylindrical coordinates. Once we obtain the upper-bounds for the growth of the self-consistent fields, then we solve the characteristic ODEs for the displacement and the velocity of a particle trajectory to obtain the displacement in terms of the velocity and the forcing effects from the fields. By writing the forcing effects in terms of the potentials, we can derive the maximum displacement for the trajectory in terms of the upper-bounds of the self-consistent fields and the external magnetic potential. Then, we can carefully determine an assumption on the magnitude of the external magnetic potential that we impose in the interior so that it overcomes the repulsive effects of the growing self-consistent fields and dominates the behavior of the particle trajectory so that the huge but \textit{finite-in-time} external magnetic potential guides the particle trajectory to the center of the potential well. It was crucial to determine the magnitude of the external magnetic potential so that it is large enough to control the particle trajectory but at the same time it is not infinite. Once we obtain the bounds for the quantities that are related to the particle trajectory, then we proceed and obtain the desired estimates for the particle distribution and the coupled fields which we use to prove the global wellposedness.

\subsection{The relativistic \textit{Vlasov-Maxwell} system} 
The two-dimensional relativistic \textit{Vlasov} equation under forcing fields $\vec{E}=(E_1,E_2)$ and $\bar{B}$ reads as 
\begin{equation}
\label{Vlasov}
\partial_t f +\hat{p}\cdot \nabla_xf+(E_1+\hat{p}_2\bar{B},E_2-\hat{p}_1\bar{B})\cdot \nabla_p f=0,
\end{equation} where $f=f(t,x,p)$ is a non-negative distribution function of a single species of charged particles at a certain time $t\in [0,T]$ for $T>0$, at a particular location $x\in\Omega\subset \rtw$, with the momentum $p\in\rtw$. 
For mathematical simplicity, we have already normalized all physical constants including the rest mass, the charge, and the speed of light to be 1 without loss of generality. Here, the velocity $\hat{p}$ is defined as $\hat{p}\eqdef \frac{p}{p^0}$, where $p^0\eqdef \sqrt{1+|p|^2}.$ Throughout the paper we assume that the spatial/physical domain $\Omega$ is a two-dimensional annulus, which can be described as 
$$\Omega\eqdef \{x\in \rtw: r_1 <|x|< r_2\},$$ for some given constants $r_1$ and $r_2$ satisfying $0<r_1<r_2.$ The magnetic field $\bar{B}$ in \eqref{Vlasov} consists of two components $B(t,x)$ and $B_{ext}(t,x)$, 
\begin{equation*}
    \bar{B} = B(t,x)+B_{ext}(t,x),
\end{equation*}
where $B_{ext}$ is an external magnetic field that will be chosen to be increasing in time and as $x$ gets closer to the boundary. The self-consistent electric field $(E_1,E_2)$ and magnetic field $B$ satisfy the following Maxwell equations:
\begin{equation}\label{Maxwell}
\begin{split}
\partial_{x_1}E_1+\partial_{x_2}E_2&=\rho,\\
\partial_tE_1& = \partial_{x_2}B-j_1,\\
\partial_tE_2& = -\partial_{x_1}B-j_2,\\
\partial_tB&=\partial_{x_2}E_1-\partial_{x_1}E_2,
\end{split}
\end{equation}
where the macroscopic charge density $\rho$ is defined as $\rho\eqdef \int_\rtw f dp$, and the $i$-th component of the current density is $j_i \eqdef \int_\rtw \hat{p}_i f dp$ for $i=1$, $2$.

We are interested in considering the \textit{Vlasov-Maxwell} system in the cylindrical-coordinates as our physical domain is a two-dimensional annulus. Therefore, we consider the change of coordinates $(x,p)\in \Omega \times \rtw\mapsto (r,\theta,p_r,p_\theta)$ where 
\begin{equation}\notag
\begin{split}
r&\eqdef \sqrt{x_1^2+x_2^2},\qquad\theta\eqdef \arctan\left(\frac{x_2}{x_1}\right),\\
p_r&\eqdef \frac{p_1x_1+p_2x_2}{\sqrt{x_1^2+x_2^2}},\text{ and }
p_\theta\eqdef \frac{p_2x_1-p_1x_2}{\sqrt{x_1^2+x_2^2}}.
\end{split}
\end{equation} 
Note that the Jacobian determinant for the changes of variables $x\mapsto (r,\theta)$ and $p\mapsto (p_r,p_\theta)$ are $r^{-1}$ and $1$, respectively. Note that $|x|=r$ and $|p|=\sqrt{p_r^2+p_\theta^2}.$
Then we obtain that the Vlasov equation \eqref{Vlasov} is now equal to 
\begin{multline}
\label{pVlasov}
\partial_t f +\hat{p}_r\partial_rf+\hat{p}_\theta \frac{1}{r}\partial_\theta f\\+\left(E_r+\hat{p}_\theta\bar{B}+\frac{p^0\hat{p}_\theta^2}{r}\right)\partial_{p_r}f+\left(E_\theta-\hat{p}_r\bar{B}-\frac{p^0\hat{p}_r\hat{p}_\theta}{r}\right)\partial_{p_\theta} f=0,
\end{multline} where $p^0=\sqrt{1+p_r^2+p_\theta^2}$, $\hat{p}_r\eqdef \frac{p_r}{p^0}$, $\hat{p}_\theta \eqdef \frac{p_\theta}{p^0}$, and $E_1\hat{e}_1+E_2\hat{e}_2=E_r\hat{r}+E_\theta\hat{\theta}$ with $\hat{e}_1\eqdef (1,0)$, $\hat{e}_2\eqdef (0,1)$, $\hat{r}\eqdef (\cos\theta, \sin\theta),$ and $\hat{\theta}\eqdef (-\sin \theta, \cos\theta)$ such that 
\begin{equation}
E_1=E_r\cos\theta-E_\theta\sin\theta\quad \text{and}\quad 
E_2=E_r\sin\theta +E_\theta\cos\theta.
\end{equation} 
This change of coordinates is standard in the nonrelativistic case, 
see \cite{Tasso2014, 2016PhDVogman} for instance. Note that the non-relativistic Vlasov equation in the cylindrical coordinates includes the additional
acceleration terms $\frac{p_\theta^2}{r}$ (the centrifugal force) and $\frac{p_rp_\theta}{r}$ (the Coriolis force).   For the relativistic case, we have one more contribution of $p^0$ in the denominators of these additional terms, and we obtain the forces $\frac{p^0\hat{p}_\theta^2}{r}$ and $\frac{p^0\hat{p}_r\hat{p}_\theta}{r}$ as in \eqref{pVlasov}. Under the same change of variables, Maxwell's equations \eqref{Maxwell} now become 
\begin{equation}\label{pMaxwell}
\begin{split}	
\frac{1}{r}\partial_r(rE_r)+\frac{1}{r}\partial_{\theta}E_\theta&=\rho,\\
\partial_tE_r& = \frac{1}{r}\partial_{\theta}B-j_r,\\
\partial_tE_\theta& = -\partial_{r}B-j_\theta,\\
\partial_tB&=\frac{1}{r}\partial_{\theta}E_r-\frac{1}{r}\partial_{r}(rE_\theta),
\end{split}
\end{equation}where $j_r$ and $j_\theta$ are the macroscopic current densities defined as $j_r\eqdef \int_\rtw \hat{p}_rfdp_rdp_\theta$ and $j_\theta\eqdef \int_\rtw \hat{p}_\theta fdp_rdp_\theta$, such that 
$$\vec{j}\eqdef (j_1,j_2)= j_r \hat{r}+j_\theta \hat{\theta}.$$

In addition, we further assume that all of the $f$, $E_r$, $E_\theta$, and $B$ are rotationally symmetric around the center of the annulus, namely \begin{equation}
\label{eqrotsym}f= f(t,r,p_r,p_\theta)\quad \text{ and }\quad
 (E_r,E_\theta,B)= (E_r,E_\theta,B)(t,r).
\end{equation}
We note that this symmetry is propagated by the Vlasov-Maxwell system \cite{MR751989}.

\subsection{Initial and boundary conditions}

We assume that $(f,E_r,E_\theta,B)$ has the following initial data of unrestricted size:
\begin{equation}\label{initial}
 \begin{split}
 f(0,r,p_r,p_\theta)&=f^0(r,p_r,p_\theta)\geq 0,\\
 E_\theta(0,r)&=E_\theta^0(r),\\
 B(0,r)&=B^0(r),\\
 E_r(0,r_1)&=\lambda\in \mathbb{R},
\end{split} 
\end{equation}
where $E^0_\theta$ and $B^0$ are given $C^1$ functions. Indeed, all the initial values of $E_r(0,r)$ for all $r\in[r_1,r_2]$ can be uniquely determined by directly integrating Gauss's law $\eqref{pMaxwell}_1$ and using the given initial conditions $\eqref{initial}_1$, $\eqref{initial}_2$, and $\eqref{initial}_4$.  For the initial distribution $f^0$ of particles, we assume that $f^0\in C^1((r_1,r_2) \times \rtw)$ and $f^0$ is compactly supported in the $r$ and $p$ variables in the following sense:  \begin{equation}\label{initialsupport}\mathrm{\mathrm{supp}}(f^0)\subseteq \left\{(x,p)\in\Omega\times\R^2: r\in I_0 \mbox{ and } |p|\leq M_0 \right\},
\end{equation}
where 
\begin{equation*}
    I_0\eqdef [r_1+\delta_0,r_2-\delta_0],
    \quad 
    \text{for some constant} 
    \quad 
    \delta_0 \in \left(0,\frac{r_2-r_1}{2}\right),
\end{equation*}
and $M_0$ is the maximal radius of the initial momentum support as 
\begin{equation}\label{initialmomentumsupport}
 M_0\eqdef \sup\{|p|: p\in \rtw \text{ and }f^0(x,p)\neq 0 \mbox{ for some }x\in\mathbb{R}^2 \} <\infty.
\end{equation}

In addition, we also assume the boundary conditions for the self-consistent fields that 
\begin{equation}\label{boundary}
 \begin{split}
 E_\theta(t,r_1)&=E_\theta^b(t,r_1),\\
 E_\theta(t,r_2)&=E_\theta^b(t,r_2),
\end{split}
\end{equation}holds where $E^b_\theta$ is given axisymmetric $C^1$ function defined on the boundary $\partial\Omega$. Then we claim that the boundary conditions \eqref{boundary} uniquely determine the boundary values $B^b$ of $B(t,r)$ at $r=r_1$ and $r=r_2$ and the boundary values do not blow up in a finite time. This will be shown in Lemma \ref{Pplus lemma} and Remark \ref{missingboundary remark}.

\begin{remark}We remark that if we further assume the boundary conditions for both $E^b_\theta$ and $B^b$ at both boundaries $r=r_1$ and $r=r_2$, then the system is over-determined. One must assume only one condition on either $E^b_\theta$ or $B^b$ for each $r=r_1$ and $r=r_2.$ Our boundary condition \eqref{boundary} is one of the possible boundary conditions, and this condition makes the calculations below the simplest due to the presence of additional $B$ on the right-hand side of \eqref{transport}. In general, we find that mixed-type boundary conditions are also fine, but we believe one should consider estimating the quantity $\|E_\theta (t)\|_{L^\infty} + \|B (t)\|_{L^\infty}$ at \eqref{Bestimate} in Proposition \ref{aprioriestimatesEB} in this case. We also note that the only mixed-type boundary condition that we do not allow is the boundary conditions for $P_+\eqdef r(E_\theta+B)$ for $r=r_2$ as in \eqref{P+} or  $P_-\eqdef r(E_\theta-B)$ for $r=r_1$ as in \eqref{P-}. In these cases, the system is again over-determined and needs a compatibility condition between the initial conditions and the boundary conditions due to the characteristic trajectory \eqref{transport}.
\end{remark}

\subsection{A finite external magnetic potential on the boundary}In this section, we introduce the external magnetic potential that we impose on the system, whose role is crucial for the magnetic confinement of the plasma. 

Before we introduce the finite time-dependent external magnetic potential $\psi_{ext}$, we first introduce an infinite potential $\psi_{base}$ which works as a prototype for the finite potential in the construction. 
The finite time-dependent external magnetic potential $\psi_{ext}$ will be constructed via the truncation of a time-independent infinite external potential $\psi_{base}$, and this will be introduced in Section \ref{confinement}. The key idea behind the construction of a time-dependent finite external potential is to establish a time-dependent \textit{moving bar} $L_{bar}(t)$ as in Hypothesis \ref{extmagassumption}. The \textit{moving bar} is growing in time and the role of it is to provide the minimal growth rate of the external potential. As long as it is larger than the maximal kinetic energy that each particle can have near the boundary, the particles are well-confined and the external potential can be finite near the boundary. This will be introduced more in detail in Section \ref{sec.truncated}.    We remark that the sufficient conditions that we require on the time-independent infinite potential $\psi_{base}$ are as follows: 
\begin{hypothesis}\label{psibasegeneral}We suppose that the time-independent magnetic potential $\psi_{base}=\psi_{base}(r)$ satisfies the following assumptions; for a given distance $\delta\in (0,\delta_0)$ from the spatial boundary $\partial\Omega$, we assume
\begin{enumerate}
	\item $\psi_{base} \in C^2((r_1+\delta,r_2-\delta)).$
	\item $\psi_{base}$ satisfies
	\[	\lim_{r\to (r_1+\delta)^+} | \psi_{base}(r) | = \lim_{r\to (r_2-\delta)^-} | \psi_{base}(r) | = \infty .	\]
\end{enumerate} In the intervals $[r_1,r_++\delta)$ and $(r_2-\delta,r_2]$, $\psi_{base}$ can take any arbitrary value.  We recall that the constant satisfies
    $\delta_0 \in \left(0,\frac{r_2-r_1}{2}\right)$.
\end{hypothesis}

\begin{remark}
Setting $|\psi_{base}|=\infty$ in these two sub-intervals $[r_1,r_++\delta)$ and $(r_2-\delta,r_2]$ is consistent with the definition in \eqref{Udelta0} and Remark \ref{rem:ConfinementDistance}, but it is not required.  One may have a small complaint on which $|\psi_{base}|=\infty$ in some open intervals is non-physical because this creates an infinitely strong external magnetic force. However, this accusation is also a fantasy, because the real/actual/physical external magnetic field (that we use) is always the $\psi_{ext}$ defined in \eqref{finitepsiext} instead of the  $\psi_{base}$, which is just a ``reference'' potential.
\end{remark}

\begin{remark}Hypothesis \ref{psibasegeneral} implies that, for any $L>0$, the set
\begin{equation}\label{UL}
	S_L \eqdef \left\{ x \in \Omega; \ |\psi_{base}(x)| \leq L \right\}
\end{equation}
is a compact, and hence it is a proper subset of the open set $\Omega\eqdef \{x\in \rtw: r_1 <|x|< r_2\}$. This will be sufficient to guarantee a positive distance away from the spatial boundary $\partial\Omega$.\end{remark}

\begin{remark}\label{remark1.4}One of the explicit examples of $\psi_{base}$ is
	\begin{equation}\label{psibasedef} \psi_{base}(r)\eqdef  \csc \left(\frac{\pi}{r_2-r_1}(r-r_1)\right)-1.
	\end{equation}
\end{remark}

Then we can construct a finite time-dependent external magnetic potential $\psi_{ext}\eqdef \psi_{ext}(t,r)$ as follows: 
\begin{hypothesis}[Hypothesis on the external magnetic potential]\label{extmagassumption}
Let us denote\\ the median radius as $r_m\eqdef \frac{r_1+r_2}{2}$. 
Then we define the external magnetic potential $\psi_{ext}=\psi_{ext}(t,r)$, using the prototype potential $\psi_{base}$ in Hypothesis \ref{psibasegeneral}, as 
\begin{equation}\label{finitepsiext}
\psi_{ext}(t,r)\eqdef \begin{cases}
&\psi_{base}(r),\ \text{if }\psi_{base}(r)\le L_{bar}(t)\\
&L_{bar}(t)+1,\ \text{if }\psi_{base}(r)\ge L_{bar}(t) +1,\\
&\text{smooth, if }L_{bar}(t)\le \psi_{base}(r)\le L_{bar}(t)+1,
\end{cases}
\end{equation}
where the \textit{moving bar} $L_{bar}(t)$ is defined as
\begin{equation}\label{movingbargeneral def}
\begin{split}
L_{bar}(t)
&\eqdef\max_{x\in U_{\delta_0}(t)} |\psi_{base}(x)|,
\end{split}
\end{equation}
where $U_{\delta_0}(t)$ is defined as 	
\begin{equation}\label{Udelta0}
U_{\delta_0}(t) \eqdef \left\{ x \in \Omega; \ |\psi_{base}(x)| \leq \frac{r_2}{r_1} \left( \max_{r\in [r_1+\delta_0,r_2-\delta_0]}|\psi_{base}(r)| \right) + \frac{K}{r_1} e^{Ct} \right\}.
\end{equation} 
Here, the initial parameter $\delta_0$ is the same constant as the one used in the definition of $I_0$ at \eqref{initialsupport}, $C$ is defined as \eqref{Cdef}, and $K$ is defined as \eqref{constant K}.
%
%
%
%
%
%
%
%
%
%
%
%
%
%
%
%
%
%
%
%

%
%
\end{hypothesis}

\begin{remark}\label{psiext remark}
   We remark that  $$\sup_{r\in [r_1,r_1+\delta_0]\cup [r_2-\delta_0,r_2]} |\psi_{ext}(t,r)| \to \infty \text{ as }t\to \infty. $$ However, $|\psi_{ext}(t,r)|$ is finite within any open interval $r\in U\subset (r_1+\delta_0,r_2-\delta_0)$ for all time $t\ge 0.$
   Moreover, it remains finite within any finite time interval $[0,T]$ for any $T>0. $
\end{remark}
\begin{remark}
For the general form of the external magnetic potential defined in Hypothesis \ref{psibasegeneral}, we can easily observe that  
	\begin{equation}\notag\begin{split}
	L_{bar}(t)=\max_{x\in U_{\delta_0}(t)} |\psi_{base}(x)|
	\le \frac{r_2}{r_1}\max_{r\in [r_1+\delta_0,r_2-\delta_0]}|\psi_{base}(r)|+ \frac{K}{r_1}e^{Ct},\end{split}
	\end{equation} by the definition of the set $U_{\delta_0}(t)$ in  \eqref{Udelta0}.
	Here
	\begin{multline}\label{constant K}
	K\eqdef \frac{\tilde{C}}{2}(r_2+r_m)(r_2-r_1)+(2r_2-r_1)\left(\frac{2\tilde{C}}{C}+M_0\right) 
	+r_2M_0+\frac{\tilde{C}r_m}{C},
	\end{multline}
	where $\delta_0$ is the same constant as the one used in the definition of $I_0$ at \eqref{initialsupport}. The constant $K$ has been determined such that \eqref{Rpsiext diff 3} holds in the arguments of using the characteristic ODEs for the particle trajectories in the proof of Lemma \ref{magconfi}.   Here, $M_0$ is the maximal radius of the initial-momentum-support defined in \eqref{initialmomentumsupport}, $C>0$ and $\tilde{C}$ are defined in \eqref{tildeCdef} and \eqref{Cdef} and depend only on $r_1,r_2,$ $\|B^0\|_{L^\infty({[r_1,r_2] })}$, $\|E_\theta^0\|_{L^\infty({[r_1,r_2] })}$, $\|E_\theta^b\|_{L^\infty([0,t]\times\partial{[r_1,r_2] })}$, $\|p^0f^0\|_{L^1({[r_1,r_2] }\times \rtw)},$ and $\lambda$.
\end{remark}

\begin{remark}\label{remark.explicit}
	For the external potential $\psi_{base}$ which is explicitly defined in \eqref{psibasedef}, we define $L_{bar}(t)$ as 
	\begin{equation}\notag\label{Mtdef1}\begin{split}
	L_{bar}(t)&\eqdef  \psi_{base} \left(r_1+\left(\frac{r_2-r_1}{\pi}\right)\arcsin(C_t)\right)
	= C_t-1,\end{split}
	\end{equation}
	where $C_t$ is defined as 
	$$C_t\eqdef 1+\frac{r_2}{r_1}\left| \csc \left(\frac{\pi}{r_2-r_1}(r_2-r_1-\delta_0)\right)-1\right|+ \frac{K}{r_1}e^{Ct},$$ and $K$ is from \eqref{constant K}. 
\end{remark}

We point out that Remark \ref{remark.explicit} is consistent with the definition \eqref{movingbargeneral def}. Note that the absolute value of the explicit magnetic potential \eqref{psibasedef} gets larger if it is closer to the boundary. So the maximum occurs at $r=r_2-\delta_0$ and we obtain Remark \ref{remark.explicit}.

\begin{remark}\label{catch22}
We would like to provide more details on the size of the finite-in-time external magnetic potential $\psi_{ext}$ with respect to the time variable $t.$  Indeed, we will compute the minimal growth rate of the external potential $\psi_{ext}$ with respect to time that we need for the magnetic confinement in \eqref{e:Upper_Bdd_for_psi(R)}, where the right-hand side of \eqref{e:Upper_Bdd_for_psi(R)} determines the size of the moving bar $L_{bar}(t)$ as in \eqref{movingbargeneral def} and \eqref{Udelta0}.

There are several crucial reasons why we can use a finite external magnetic potential $\psi_{ext}$ to confine the plasmas.  The main observation is that the a priori estimates for the self-consistent electro-magnetic fields $E_r,$ $E_\theta$ and $B$ in Proposition \ref{E1estimate} and Proposition \ref{aprioriestimatesEB} are independent of the external magnetic field $B_{ext}$ (or equivalently, the potential $\psi_{ext}$). As a result, our choice of the finite barrier (i.e., $\psi_{ext}$) will not affect the velocity control \eqref{Vbound} in Lemma \ref{Vbound lemma}, which is a direct consequence of the estimates on $E_r$, $E_{\theta}$ and $B$ indeed. The crucial estimate \eqref{e:Upper_Bdd_for_psi(R)} follows directly from the velocity bound \eqref{Vbound}, and hence it is also independent of the choice of $\psi_{ext}$. In other words, having such an $L^\infty$ velocity control \eqref{Vbound} that is independent of the external magnetic potential $\psi_{ext}$ is the crucial reason why we are able to confine the plasma by using a finite magnetic potential. 
Thus, it is crucial to note that such a circular reasoning or a \textit{catch-22} situation where a stronger $\psi_{ext}$ may also speed up the particles and hence a even stronger $\psi_{ext}$ would be needed to confine the plasma does not appear in the analysis.

The observation in the physical side is also interesting. Physically, the external magnetic field and its potential only affect the plasma uniformly, but will not affect the self-interactions among particles. As a result, the external magnetic field can be used to \textit{move} the particles as in the process of confinement, but it cannot affect the self-consistent electric and magnetic fields in general.
\end{remark}

In Section \ref{confinement}, we will prove that both $\psi_{base}$ and $\psi_{ext}$ can be used as an external magnetic potential of the system so that all the charged particles can be confined globally in time. This will prove Theorem \ref{maintheorem2}.

\begin{remark}
	It turned out that the Vlasov-Maxwell equations in cylindrical coordinates contains additional forcing terms in the equation which are purely formed by the coordinate changes; indeed, those additional terms are related to centrifugal and Coriolis forces, which only appear in the rotating frame. However, these additional terms create extra singularities when we implement the previously existing argument. Due to the extra inhomogeneity from the magnetic field that appears on the right hand side of \eqref{sum1} below, the fields and their derivatives have higher growth and we needed to control the additional growth via considering a sufficiently large but finite-in-time external magnetic potential well. 
\end{remark}
\subsection{Main results}We now state our main theorems. The first theorem that we state is on the global well-posedness on the Cauchy problem to the relativistic \textit{Vlasov-Maxwell} system in an annulus:

\begin{theorem}[Global well-posedness of the Cauchy problem]\label{maintheorem1}
Suppose that Hypothesis \ref{extmagassumption} and the rotational symmetry \eqref{eqrotsym} hold. Define the external magnetic field $B_{ext}=B_{ext}(t,r)$  as \begin{equation}\label{Bextdef}B_{ext}(t,r)\eqdef \frac{1}{r}\frac{\partial(r\psi_{ext}(t,r))}{\partial r}.\end{equation}  For some constants $\delta_0 \in(0,\frac{r_2-r_1}{2})$
 and $M_0>0$, we assume that $f^0\in C^1((r_1,r_2) \times \rtw)$ and $f^0$ is compactly supported in the $r$ and $p$ variables in the sense of \eqref{initialsupport} and \eqref{initialmomentumsupport}.  Suppose that $E^0_\theta$ and $B^0$ are $C^1((r_1,r_2) )$ functions. Then there exists a unique non-negative $C^1([0,\infty)\times (r_1,r_2) \times \rtw)$ solution $f$ to the relativistic \textit{Vlasov-Maxwell} system \eqref{pVlasov} and \eqref{pMaxwell} subject to the initial condition \eqref{initial} and the boundary condition \eqref{boundary}.
\end{theorem}
Additionally, we introduce our main theorem on the confinement of the plasma by a finite external magnetic field at the boundary:

\begin{theorem}[Global confinement of the plasma]\label{maintheorem2} Let $r_m\eqdef \frac{r_1+r_2}{2}.$ Suppose that the support condition \eqref{initialsupport} and \eqref{initialmomentumsupport} for the initial condition holds for some constants $\delta_0\in(0,\frac{r_2-r_1}{2})$ and $M_0>0$.  Suppose that Hypothesis \ref{extmagassumption} and the rotational symmetry \eqref{eqrotsym} hold for a given $\delta\in(0,\delta_0).$ Define the external magnetic field $B_{ext}=B_{ext}(t,r)$  as in \eqref{Bextdef}. Then the unique $C^1$ solution $f(t,r,p_r,p_\theta)$ obtained in Theorem \ref{maintheorem1} satisfies 
$$\mathrm{dist}(\mathrm{supp}_{x} (f)(t), \ \partial \Omega)>\delta>0,$$
for any $t\in [0,\infty)$, where $\mathrm{supp}_{x}(f)(t)$ is defined as
$$\mathrm{supp}_{x}(f)(t)\eqdef \{x\in \Omega\  |\  f(t,r,p_r,p_\theta)\neq 0,  \text{ for some } (p_r,p_\theta)\in\rtw, \text{ where }r\eqdef |x| \}.$$
\end{theorem}

\begin{remark}\label{rem:ConfinementDistance}
	We remark that, by choosing \Blue{$\psi_{base}$ and } $\psi_{ext}$ appropriately as in Hypotheses \ref{psibasegeneral} and \ref{extmagassumption}, we are indeed able to confine the plasma in a given compact set $$\{x\in \Omega \ |\   r_1+\delta\le |x|\le r_2-\delta\},$$ for any given time $t\in [0,\infty)$ as we have $$\mathrm{dist}(U_{\delta_0}(t),\partial\Omega)>\delta >0,$$ by Hypothesis \ref{psibasegeneral} where $U_{\delta_0}(t)$ is defined as in \eqref{Udelta0}.
	This set contains the initial spatial support.
\end{remark}

\begin{remark}
In this paper, we implement the magnetic confinement in a compact set using the full relativistic Vlasov-Maxwell system. We use a finite external potential to confine the plasma and can control the size of the spatial support of the plasma as we desire as long as it includes the initial spatial support.

    There are many results on the Vlasov-Poisson system with given external magnetic field (both stationary and nonstationary, up to the full 3-dimensional) such as \cite{MR2997588,MR3723427,MR3318184,MR3148082,MR3568167,MR3485969,MR2733252,MR3709891,Belyaeva_2021}.     However, there are very few results regarding magnetic confinement for the relativistic Vlasov-Maxwell system. Including ours, there are only three results to the best of our knowledge. Others are \cite{MR3294216,MR3782409}. 
Though the work \cite{MR3294216,MR3404774} considers the confinement in the 1.5-dimensional domain with $x_1\in [0,1]$ and $v \in \mathbb{R}^2$, their confinement is indeed in a 2-dimensional infinite strip with the symmetry in $x_2$ variable which is not compact.

The implementation of the magnetic confinement in a compact domain is relevant to heating in tokamaks \cite{cottrell1988superthermal,goede1976ion},
magnetic mirror-confined plasma \cite{crawford1965review,hubbard1978electrostatic,perraut1982systematic,post1966electrostatic}, and electron cyclotron resonance heating \cite{guest2009electron}.
\end{remark}

\subsection{Organization of the paper}
In order to prove our main theorems, we first need to obtain a priori $L^\infty$ estimates for the self-consistent fields and the particle distribution to the system. We first obtain the estimates for the fields in Section~\ref{sec:fieldest} via applying the method of characteristics to the wave equations. Then, based on the field-estimates, we start proving our main theorem on the magnetic confinement, Theorem~\ref{maintheorem2}, in Section~\ref{confinement}. In order to prove the global existence and the uniqueness of a classical solution in Section~\ref{sec:global}, we make several estimates in Section~\ref{Linftysection} and Section~\ref{sec:derivatives} on the electro-magnetic fields and the distribution. In Section~\ref{Linftysection}, we also obtain the $L^\infty$-moment propagation of the solution, and the $L^\infty$ estimates for the macroscopic mass density and the current density. Since we are interested in constructing $C^1$ solutions to the system, we also need to obtain the $L^\infty$ estimates of the first-order derivatives of the fields and the distribution. This is done in Section~\ref{sec:derivatives}. Finally, we use the a priori estimates and the iteration argument to prove the existence, the uniqueness, and the non-negativity of a global $C^1$ solution to the \textit{Vlasov-Maxwell} system in Section~\ref{sec:global}.

\section{A priori estimates for the self-consistent fields} \label{sec:fieldest}

In the forthcoming sections we will obtain some uniform a priori estimates for $(f,E_r,E_\theta,B)$. Consider $C^1$ solutions $(f,E_r,E_\theta,B)$ to \eqref{pVlasov}-\eqref{boundary} on a finite time interval $[0,T]$. We a priori let the particles be confined as in Theorem \ref{maintheorem2} throughout this section.
\subsection{Estimates of the field \texorpdfstring{$E_r$}{}}
We start with introducing the upper bound for $E_r$ in this section.
  More precisely, we have the following proposition:
  \begin{proposition}\label{E1estimate}
  	We have $$\|E_r\|_{L^\infty([0,T]\times [r_1,r_2] )}\leq \|f^0\|_{L^1([r_1,r_2] \times \rtw)}+\lambda.$$
  \end{proposition}
\begin{proof}  
By integrating the continuity equation
\begin{equation}\label{continuity} 			
	\partial_t\rho +\frac{1}{r}\partial_r(rj_r)= 0 
\end{equation}with respect to $r\ dr  dt$ and using that $j_r=0$ on the boundaries, we have the conservation of total charge: for any $t\in [0,T]$,
$$\int_{r_1}^{r_2 }  \rho(t,r)rdr  = \int_{r_1}^{r_2 }  \rho(0,r)rdr  
 = \|f^0\|_{L^1({[r_1,r_2] }\times \rtw)}.$$
On the other hand, it follows from Gauss's law $\eqref{pMaxwell}_1$ that we also have $$\partial_r (rE_r)=r\rho,$$ and hence, for any $R\in [r_1, r_2]$,
$$RE_r(t,R)-r_1E_r(t,r_1)=\int_{r_1}^Rr\rho(t,r) dr.$$
  It follows from Amp\`{e}re's circuital law $\eqref{pMaxwell}_2$ that $\partial_t E_r=-j_r$, so we further have
   $$E_r(t,r_1)=E_r(0,r_1)-\int_0^t j_r|_{r=r_1} d\tau =\lambda -\int_0^t j_r(\tau, r_1)d\tau=\lambda,$$ since $j_r|_{r=r_1} \equiv 0$. Therefore, we finally obtain
   $$E_r(t,R)=\frac{1}{R}\int_{r_1}^Rr\rho dr +\frac{r_1}{R}\lambda,$$ for all $r_1\leq R\leq r_2$.
   This implies that
   $$\|E_r\|_{L^\infty([0,T]\times {[r_1,r_2] })}\leq \|\rho^0\|_{L^1({[r_1,r_2] })}+\lambda=\|f^0\|_{L^1({[r_1,r_2] }\times \rtw)}+\lambda.$$
   This completes the proof.
   \end{proof}
   
\subsection{Estimates of the fields \texorpdfstring{$E_\theta$}{} and \texorpdfstring{$B$}{}}\label{estiEthetaB}
  In this section, we use the method of  characteristics to estimate the fields $E_\theta$ and $B$. First of all, we consider the third and the fourth equations of \eqref{pMaxwell}. We multiply the third and the fourth equations by $r$ and obtain
   \begin{equation}
   \label{sum1}\partial_t (rE_\theta)+\partial_r(rB)=B-rj_\theta,
   \end{equation}
   and 
   \begin{equation}
   \label{sum2}
   \partial_t(rB)+\partial_r(rE_\theta)=0.
   \end{equation}
Hence, it follows from direct addition and subtraction that
  \begin{equation}\label{transport}
  \partial_t(rE_\theta\pm rB)\pm\partial_r(rE_\theta\pm rB)=B-rj_\theta.
  \end{equation}
  
  Define $P_{\pm}\eqdef rE_\theta\pm rB$ and fix $t\in (0,T]$ and $r\in {[r_1,r_2] }$. We will use the fact that the solutions of the transport equations \eqref{transport} at $(t,r)$ are affected only by the values inside the characteristic cone.
Therefore, we have
\begin{equation}\label{P+}
	P_+(t,r)=P_+(t_1(t,r),r-t+t_1(t,r))+\int_{t_1(t,r)}^t(B-rj_\theta)(\tau,r-t+\tau)d\tau,
\end{equation}
and
\begin{equation}\label{P-}
	P_-(t,r)=P_-(t_2(t,r),r+t-t_2(t,r))+\int_{t_2(t,r)}^t(B-rj_\theta)(\tau,r+t-\tau)d\tau,
\end{equation}where we define $t_1=t_1(t,r)\eqdef \max\{0,t-r+r_1\}$ and $t_2=t_2(t,r)\eqdef \max\{0,t-r_2+r\}.$
Then, it follows from direct addition and subtraction that 
\begin{multline}\label{e:Char_Formula_for_E_theta}
 (rE_\theta)(t,r)=\frac{1}{2}\left(P_+(t_1(t,r),r-t+t_1(t,r))+P_-(t_2(t,r),r+t-t_2(t,r))\right)\\+\frac{1}{2}\int_{t_1(t,r)}^t(B-rj_\theta)(\tau,r-t+\tau)d\tau+\frac{1}{2}\int_{t_2(t,r)}^t\left[(B-rj_\theta)(\tau,r+t-\tau)\right]d\tau,
\end{multline} 
 and
\begin{multline}\label{e:Char_Formula_for_B}
 (rB)(t,r)=\frac{1}{2}\left(P_+(t_1(t,r),r-t+t_1(t,r))-P_-(t_2(t,r),r+t-t_2(t,r))\right)\\+\frac{1}{2}\int_{t_1(t,r)}^t(B-rj_\theta)(\tau,r-t+\tau)d\tau-\frac{1}{2}\int_{t_2(t,r)}^t\left[(B-rj_\theta)(\tau,r+t-\tau)\right]d\tau.
\end{multline}
Therefore, we need to estimate the upper-bounds of the following two integrals:\begin{equation}\label{int.jtheta}\int_{t_1(t,r)}^t(rj_\theta)(\tau,r-t+\tau)d\tau\text{ and }\int_{t_2(t,r)}^t(rj_\theta)(\tau,r+t-\tau)d\tau.\end{equation}

We are now ready to state our main lemma of this section. The following lemma is on the upper-bounds of the sum of the two integrals of our interest from the argument above. We will use this upper-bound estimate to bound our fields $E_\theta$ and $B$ later in this section. 
 \begin{lemma}\label{jtheta}
 Let $t\in (0,T]$. Suppose that $$\lim\limits_{|p|\to\infty}f=0.$$ Then if $r<\frac{r_1+r_2}{2}$, we have 
 \begin{multline}\label{2.10}\int_{t_1(t,r)}^t(r|j_\theta|)(\tau,r-t+\tau)d\tau+
 \int_{t_2(t,r)}^t(r|j_\theta|)(\tau,r+t-\tau)d\tau\\\leq \int_{r_1}^{r_2 }  r'e(t_2(t,r),r')dr' +\int_{t_2(t,r)}^{t_1(t,r)}r_1(E^b_\theta B^b)(\tau,r_1)d\tau,
 \end{multline}where $$e(t,r)\eqdef \frac{1}{2}(|E(t,r)|^2+B^2(t,r))+\int_\rtw p^0f(t,r,p_r,p_\theta)dp_rdp_\theta.$$ On the other hand, if $r\ge \frac{r_1+r_2}{2}$, then we instead have
 \begin{multline}\label{2.11}\int_{t_1(t,r)}^t(r|j_\theta|)(\tau,r-t+\tau)d\tau+
 \int_{t_2(t,r)}^t(r|j_\theta|)(\tau,r+t-\tau)d\tau\\\leq \int_{r_1}^{r_2 }  r'e(t_1(t,r),r')dr' +\int_{t_2(t,r)}^{t_1(t,r)}r_2(E^b_\theta B^b)(\tau,r_2)d\tau.
 \end{multline}
 \end{lemma}
 The proof for Lemma \ref{jtheta} heavily relies on the following identity. The identity \eqref{dtedrmeq} that we will introduce in the following lemma is the energy identity and this is related to Poynting's theorem.
 \begin{lemma}\label{dtedrm}
 Define $$e(t,r)\eqdef \frac{1}{2}(|E(t,r)|^2+B^2(t,r))+\int_\rtw p^0f(t,r,p_r,p_\theta)dp_rdp_\theta$$ and $$m(t,r)\eqdef \int_\rtw p_rf(t,r,p_r,p_\theta)dp_rdp_\theta+(E_\theta B)(t,r).$$ Suppose that $$\lim\limits_{|p|\to\infty}f=0.$$Then we have \begin{equation}\label{dtedrmeq}
 \partial_t e+\frac{1}{r}\partial_r(rm)= 0.
 \end{equation}
 \end{lemma}

 \begin{proof}[Proof of Lemma \ref{dtedrm}]
 It follows from Maxwell's equations \eqref{pMaxwell} that $$
 	\partial_t e=\int_\rtw p^0(\partial_t f)dp_rdp_\theta -\int_\rtw (\hat{p}_r,\hat{p}_\theta )\cdot (E_r,E_\theta)f dp_rdp_\theta-E_\theta (\partial_r B)-\frac{B}{r}\partial_r (rE_\theta),
 $$
 and
 $$ \frac{1}{r}\partial_r(rm)=\frac{1}{r}\int_\rtw p_r fdp_rdp_\theta+\int_\rtw p_r\partial_r f dp_rdp_\theta+\frac{B}{r}\partial_r(rE_\theta)+E_\theta (\partial_rB).$$
 Therefore, we have
 \begin{multline}\label{continuity identity}
 	\partial_t e+ \frac{1}{r}\partial_r(rm)=\int_\rtw p^0(\partial_t f)dp_rdp_\theta +\int_\rtw p_r\partial_r f dp_rdp_\theta\\-\int_\rtw (\hat{p}_r,\hat{p}_\theta )\cdot (E_r,E_\theta)f dp_rdp_\theta+\frac{1}{r}\int_\rtw p_r fdp_rdp_\theta.
 \end{multline}
 By \eqref{pVlasov}, we further have
 \begin{equation}\label{vlasovuse}
 p^0(\partial_t f +\hat{p}_r\partial_rf)= -p^0\left(E_r+\hat{p}_\theta \bar{B}+\frac{p^0\hat{p}_\theta^2}{r}, E_\theta -\hat{p}_r\bar{B}-\frac{p^0\hat{p}_r\hat{p}_\theta}{r}\right)\cdot (\partial_{p_r}f,\partial_{p_\theta} f).
 \end{equation}
 Note that 
 \begin{multline*}
 	-\int_\rtw  p^0(E_r, E_\theta)\cdot (\partial_{p_r}f,\partial_{p_\theta} f)dp_rdp_\theta\\
 	=\int_\rtw  \{\partial_{p_r}(p^0E_r)f+\partial_{p_\theta}(p^0E_\theta)f \}dp_rdp_\theta
 	=\int_\rtw (\hat{p}_r,\hat{p}_\theta )\cdot (E_r,E_\theta)f dp_rdp_\theta,
 \end{multline*}  where we use the integral by parts and that $p^0f$ is vanishing at $p_r=+\infty$ for the first identity. 
 Plugging this and \eqref{vlasovuse} into \eqref{continuity identity}, we have
  \begin{multline*}
 \partial_t e + \frac{1}{r}\partial_r(rm)\\=\frac{1}{r}\int_\rtw p_r fdp_rdp_\theta-\int_\rtw \left((p_\theta \bar{B}+\frac{p_\theta^2}{r},  -p_r\bar{B}-\frac{p_rp_\theta}{r})\cdot (\partial_{p_r}f,\partial_{p_\theta} f)\right) dp_rdp_\theta.
 \end{multline*}
 Applying integration by parts to the last integral, we finally obtain \eqref{dtedrmeq}.
 \end{proof}
 
Now we are ready to prove Lemma \ref{jtheta}.
 \begin{proof}[Proof of Lemma~\ref{jtheta}]
Recall that $t_1(t,r)\eqdef \max\{0,t-r+r_1\}$ and $t_2(t,r)\eqdef \max\{0,t-r_2+r\}.$ 
 We first consider the case that $r_1\le r<\frac{r_1+r_2}{2}.$ In this case, note that $t_1(t,r)>t_2(t,r).$ For any fixed $\theta\in[0,2\pi),$ we consider the two-dimensional space-time region $\Delta\eqdef \Delta_1\cup \Delta_2$ as in Figure \ref{fig 1} where
 $$\Delta_1\eqdef \{(\tau,r'): t_1(t,r)\leq \tau\leq t\ \text{and}\ |r'-r|\leq t-\tau\}$$ and $$\Delta_2\eqdef \{(\tau,r'): t_2(t,r)\leq \tau\leq t_1(t,r) \  \text{and}\ r_1\leq r'\leq r+t-\tau\}.$$ 
 \begin{figure}[h]
 	\includegraphics[scale=0.2]{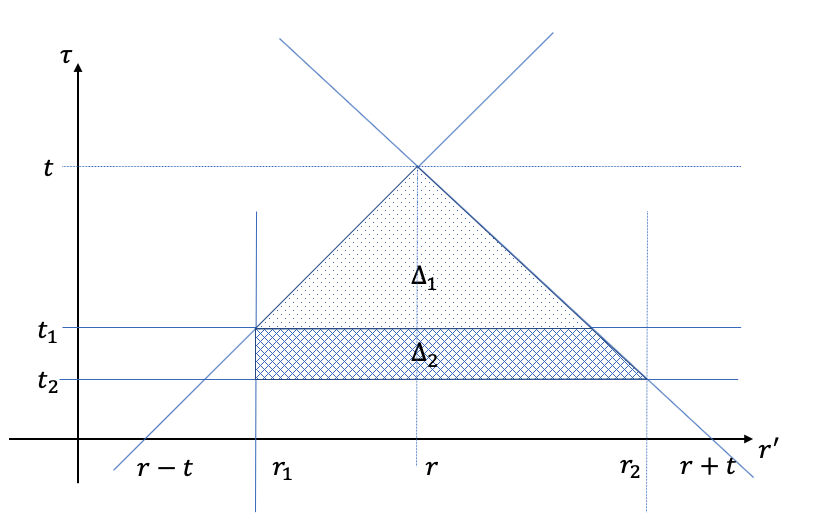}
 	\caption{The domain $\Delta$}\label{fig 1}
 \end{figure}\newline
Integrating \eqref{dtedrmeq} over $\Delta$ with respect to the measure $rdrdt$ and then applying Green's theorem with the counterclockwise line integral, we obtain
 \begin{multline*}
0=\Blue{\int_\Delta [\partial_{r'}(r'm)+\partial_\tau(r'e)]dr'd\tau=\oint_{\partial\Delta}\left(-(r'e)dr'+(r'm)d\tau\right)}\\ 
 =-\int_{r_1}^{r_2}(r'e)(t_2(t,r),r')dr'+\int_{t_2(t,r)}^t(rm+re)(\tau,r+t-\tau)d\tau\\+\int_{t}^{t_1(t,r)}(rm-re)(\tau,r-t+\tau)d\tau+\int_{t_1(t,r)}^{t_2(t,r)}(rm)(\tau,r_1)d\tau,
 \end{multline*}where the first integral is the integration on the bottom line, the second and the third ones are the integration at the top sides of the cone, and the last one is the integral on the vertical boundary at $r=r_1$.  Then a direct rearrangement yields 
  \begin{multline}\label{219}
\int_{t_2(t,r)}^t(re+rm)(\tau,r+t-\tau)d\tau+\int^{t}_{t_1(t,r)}(re-rm)(\tau,r-t+\tau)d\tau\\ =\int_{r_1}^{r_2}(r'e)(t_2(t,r),r')dr'+\int_{t_2(t,r)}^{t_1(t,r)}r_1(E^bB^b)(\tau,r_1)d\tau,
  \end{multline} because $m(\tau,r_1)=(E^bB^b)(\tau,r_1)$ by the boundary condition \eqref{boundary}. It follows from the definitions of $e$ and $m$ that 
  \begin{multline*}
  	(re\pm rm)(t,r)=\frac{r}{2}[|E|^2+B^2]+r\int_\rtw p^0f(t,r,p_r,p_\theta)dp_rdp_\theta
  	\pm (rE_\theta B)(t,r)\\
  	\pm r \int_\rtw (p_rf)(t,r,p_r,p_\theta)dp_rdp_\theta\\
  	=\frac{r}{2}[|E_r|^2+(E_\theta\pm B)^2]+r\int_\rtw (p^0\pm p_r)f(t,r,p_r,p_\theta)dp_rdp_\theta
  	\\\geq r\int_\rtw \frac{|p_\theta|}{p^0}f(t,r,p_r,p_\theta)dp_rdp_\theta,
  \end{multline*}where the last inequality holds as $f$ is non-negative and $p^0\pm p_r \ge \frac{|p_\theta|}{p^0}.$
  Together with \eqref{219}, we finally obtain \begin{multline}\int_{t_1(t,r)}^t(r|j_\theta|)(\tau,r-t+\tau)d\tau+
  \int_{t_2(t,r)}^t(r|j_\theta|)(\tau,r+t-\tau)d\tau\\\leq  \int_{r_1}^{r_2 } r'e(t_2(t,r),r')dr' +\int_{t_2(t,r)}^{t_1(t,r)}r_1(E^b_\theta B^b)(\tau,r_1)d\tau.
  \end{multline}

 On the other hand, if  $\frac{r_1+r_2}{2}\le r\le r_2,$ then $t_1(t,r)\le t_2(t,r).$ For any fixed $\theta\in[0,2\pi),$ we consider the two-dimensional space-time region $\Delta\eqdef \Delta_1\cup \Delta_2$ where
 $$\Delta_1\eqdef \{(\tau,r'): t_2(t,r)\leq \tau\leq t\ \text{and}\ |r'-r|\leq t-\tau\}$$ and $$\Delta_2\eqdef \{(\tau,r'): t_1(t,r)\leq \tau\leq t_2(t,r) \  \text{and}\ r+t-\tau\le r'\le r_2\}.$$ 
 Therefore, the counterclockwise line integral for the Green theorem is the same as the case $r\le \frac{r_1+r_2}{2}$ except that now the line integral $-\int_{r_1}^{r_2}(r'e)(t_1(\tau,r),r')dr$ is along the line $\tau=t_1$ and the line integral $\int_{t_1(t,r)}^{t_2(t,r)}(rm)(\tau,r')d\tau$ is now along the line $r'=r_2$, instead. Therefore we obtain \begin{multline}\int_{t_1(t,r)}^t(r|j_\theta|)(\tau,r-t+\tau)d\tau+
  \int_{t_2(t,r)}^t(r|j_\theta|)(\tau,r+t-\tau)d\tau\\\leq  \int_{r_1}^{r_2 } r'e(t_1(t,r),r')dr' +\int_{t_2(t,r)}^{t_1(t,r)}r_2(E^b_\theta B^b)(\tau,r_2)d\tau.
  \end{multline}
 \end{proof}
We are now interested in deriving an upper-bound estimate for the energy  $$\int_{r_1}^{r_2 }  r'e(\min\{t_1(t,r),t_2(t,r)\},r')dr' $$ that appeared in the proof of Lemma \ref{jtheta}. 
Indeed, we have the following lemma on the conservation of the energy. 
 \begin{lemma}\label{energydissipation}Define \Blue{$e(t,r)$ as in Lemma \ref{dtedrm}.} Then for any $t\in[0,T]$, 
 	\begin{multline*}
 	\int_{r_1}^{r_2 }  r'e(t,r')dr' \\ =\int_{r_1}^{r_2 }  r'e(0,r')dr' -\int_0^t\left[(r_2E^b_\theta B^b)(\tau,r_2)-(r_1E^b_\theta B^b)(\tau,r_1)\right]d\tau.
 	\end{multline*}
 \end{lemma}

\begin{proof}[Proof of Lemma \ref{energydissipation}]
We observe that
$$\partial_t\int_{r_1}^{r_2 }   r'e(t,r')dr' =\int_{r_1}^{r_2 }   \partial_t (r'e)(t,r')dr' = -\int_{r_1}^{r_2 }   \partial_{r'} (r'm)(t,r')dr' ,$$
because of Identity~\eqref{dtedrmeq}. Then we further have
$$\int_{r_1}^{r_2 }   \partial_{r'} (r'm)(t,r')dr' = r_2m(t,r_2)-r_1m(t,r_1).$$ Since $f$ vanishes at the boundaries $r=r_1$, $r_2$, the definition of $m$ further implies that 
$$ r_2m(t,r_2)-r_1m(t,r_1)=\left[(r_2E^b_\theta B^b)(t,r_2)-(r_1E^b_\theta B^b)(t,r_1)\right].$$
Therefore, we obtain the lemma by integrating with respect to the time variable.
\end{proof}
Finally, in the following lemma, we obtain the following identity for the boundary values $P_+(t,r)$ at $r=r_1$ and $r=r_2$. We can obtain almost the same lemma for $P_-$ and we omit it.
\begin{lemma}
\label{Pplus lemma}Define $M\eqdef \left\lfloor\frac{t}{r_2-r_1}\right\rfloor\ge 0.$
For any $t\in [0,T],$ we have,
for even $M$,
     \begin{multline}\notag
        P_+(t,r_1)
        =-P_-(0,t-M(r_2-r_1)+r_1)\\-\int_{0}^{t-M(r_2-r_1)}(B-r_1j_\theta)(\tau,r_1+t-M(r_2-r_1)-\tau)d\tau\\+2r_1\sum_{k=0}^{M/2}E_\theta^b(t-2k(r_2-r_1),r_1)\\
        -2r_21_{M\ge 2}\sum_{k=0}^{M/2-1}E_\theta^b(t-(2k+1)(r_2-r_1),r_2)
        \\+1_{M\ge 2}\sum_{k=1}^{M/2}\int_{t-2k(r_2-r_1)}^{t-(2k-1)(r_2-r_1)}(B-r_2j_\theta)(\tau,r_2-t+(2k-1)(r_2-r_1)+\tau)d\tau\\
        -1_{M\ge 2}\sum_{k=0}^{M/2-1}\int^{t-2k(r_2-r_1)}_{t-(2k+1)(r_2-r_1)}(B-r_1j_\theta)(\tau,r_1+t-2k(r_2-r_1)-\tau)d\tau,
    \end{multline} and for odd $M$,
     \begin{multline}\notag
        P_+(t,r_1)
        =P_+(0,r_2-t+M(r_2-r_1))\\+\int_{0}^{t-M(r_2-r_1)}(B-r_2j_\theta)(\tau,r_2-t+M(r_2-r_1)+\tau)d\tau\\+2r_1\sum_{k=0}^{\frac{M-1}{2}}E_\theta^b(t-2k(r_2-r_1),r_1)\\
        -2r_2\sum_{k=0}^{\frac{M-1}{2}}E_\theta^b(t-(2k+1)(r_2-r_1),r_2)
        \\+1_{M\ge 3}\sum_{k=1}^{\frac{M-1}{2}}\int_{t-2k(r_2-r_1)}^{t-(2k-1)(r_2-r_1)}(B-r_2j_\theta)(\tau,r_2-t+(2k-1)(r_2-r_1)+\tau)d\tau\\
        -\sum_{k=0}^{\frac{M-1}{2}}\int^{t-2k(r_2-r_1)}_{t-(2k+1)(r_2-r_1)}(B-r_1j_\theta)(\tau,r_1+t-2k(r_2-r_1)-\tau)d\tau.
    \end{multline}Also, we will obtain similar representations for $P_+(t,r_2)$ by using
      $$P_+(t,r_2)=P_+(t-(r_2-r_1),r_1)+\int_{t-(r_2-r_1)}^t(B-r_2j_\theta)(\tau,r_2-t+\tau)d\tau,$$ if $t\ge r_2-r_1$ and 
  $$	P_+(t,r_2)=P_+(0,r_2-t)+\int_{0}^t(B-r_2j_\theta)(\tau,r_2-t+\tau)d\tau,$$ if $0\le t<r_2-r_1.$
\end{lemma}
\begin{proof}
     We first observe that
    $$P_+(t,r_1)= -P_-(t,r_1)+2r_1E_\theta^b(t,r_1).$$ Then by \eqref{P-} we have \begin{multline}\notag
        P_+(t,r_1)= -P_-(t-(r_2-r_1),r_2)\\
        -\int_{t-(r_2-r_1)}^t(B-r_1j_\theta)(\tau,r_1+t-\tau)d\tau+2r_1E_\theta^b(t,r_1).
    \end{multline}Then using 
    $$P_-(t-(r_2-r_1),r_2)=-P_+(t-(r_2-r_1),r_2)+2r_2E_\theta^b(t-(r_2-r_1),r_2)$$ and using \eqref{P+}, we have
    \begin{multline}\notag
        P_+(t,r_1)= -P_-(t-(r_2-r_1),r_2)\\
        -\int_{t-(r_2-r_1)}^t(B-r_1j_\theta)(\tau,r_1+t-\tau)d\tau+2r_1E_\theta^b(t,r_1)\\
        =P_+(t-(r_2-r_1),r_2)-2r_2E_\theta^b(t-(r_2-r_1),r_2)\\-\int_{t-(r_2-r_1)}^t(B-r_1j_\theta)(\tau,r_1+t-\tau)d\tau+2r_1E_\theta^b(t,r_1),
    \end{multline}
    and so
        \begin{multline}\notag
        P_+(t,r_1)
        =P_+(t-2(r_2-r_1),r_1)+\int_{t-2(r_2-r_1)}^{t-(r_2-r_1)}(B-r_2j_\theta)(\tau,r_2-t+(r_2-r_1)+\tau)d\tau\\
        -2r_2E_\theta^b(t-(r_2-r_1),r_2)\\-\int_{t-(r_2-r_1)}^t(B-r_1j_\theta)(\tau,r_1+t-\tau)d\tau+2r_1E_\theta^b(t,r_1)\\
        =-P_-(t-2(r_2-r_1),r_1)+2r_1E_\theta^b(t-2(r_2-r_1),r_1)\\+\int_{t-2(r_2-r_1)}^{t-(r_2-r_1)}(B-r_2j_\theta)(\tau,r_2-t+(r_2-r_1)+\tau)d\tau\\
        -2r_2E_\theta^b(t-(r_2-r_1),r_2)\\-\int_{t-(r_2-r_1)}^t(B-r_1j_\theta)(\tau,r_1+t-\tau)d\tau+2r_1E_\theta^b(t,r_1).
    \end{multline}
    Repeating this procedure of reducing the time variable $M-2$ more times, we obtain for even $M$
     \begin{multline}\notag
        P_+(t,r_1)
        =-P_-(t-M(r_2-r_1),r_1)+2r_1\sum_{k=0}^{M/2}E_\theta^b(t-2k(r_2-r_1),r_1)\\
        -2r_2\sum_{k=0}^{M/2-1}E_\theta^b(t-(2k+1)(r_2-r_1),r_2)
        \\+\sum_{k=1}^{M/2}\int_{t-2k(r_2-r_1)}^{t-(2k-1)(r_2-r_1)}(B-r_2j_\theta)(\tau,r_2-t+(2k-1)(r_2-r_1)+\tau)d\tau\\
        -\sum_{k=0}^{M/2-1}\int^{t-2k(r_2-r_1)}_{t-(2k+1)(r_2-r_1)}(B-r_1j_\theta)(\tau,r_1+t-2k(r_2-r_1)-\tau)d\tau,
    \end{multline} and for odd $M$
     \begin{multline}\notag
        P_+(t,r_1)
        =P_+(t-M(r_2-r_1),r_2)+2r_1\sum_{k=0}^{\frac{M-1}{2}}E_\theta^b(t-2k(r_2-r_1),r_1)\\
        -2r_2\sum_{k=0}^{\frac{M-1}{2}}E_\theta^b(t-(2k+1)(r_2-r_1),r_2)
        \\+1_{M\ge 3}\sum_{k=1}^{\frac{M-1}{2}}\int_{t-2k(r_2-r_1)}^{t-(2k-1)(r_2-r_1)}(B-r_2j_\theta)(\tau,r_2-t+(2k-1)(r_2-r_1)+\tau)d\tau\\
        -\sum_{k=0}^{\frac{M-1}{2}}\int^{t-2k(r_2-r_1)}_{t-(2k+1)(r_2-r_1)}(B-r_1j_\theta)(\tau,r_1+t-2k(r_2-r_1)-\tau)d\tau.
    \end{multline}
    Finally, using \eqref{P+} and \eqref{P-} we can write $P_-(t-M(r_2-r_1),r_1)$ and  $P_+(t-M(r_2-r_1),r_2)$ in each case in terms of the initial data and the integrals as
    \begin{multline*}P_-(t-M(r_2-r_1),r_1)=P_-(0,t-M(r_2-r_1)+r_1)\\+\int_{0}^{t-M(r_2-r_1)}(B-r_1j_\theta)(\tau,r_1+t-M(r_2-r_1)-\tau)d\tau,\end{multline*} and 
  \begin{multline*}P_+(t-M(r_2-r_1),r_2)
  =P_+(0,r_2-t+M(r_2-r_1))\\+\int_{0}^{t-M(r_2-r_1)}(B-r_2j_\theta)(\tau,r_2-t+M(r_2-r_1)+\tau)d\tau.\end{multline*}
  Therefore, we obtain the lemma for $P_+(t,r_1).$ Also, note that by \eqref{P+}
  $$P_+(t,r_2)=P_+(t-(r_2-r_1),r_1)+\int_{t-(r_2-r_1)}^t(B-r_2j_\theta)(\tau,r_2-t+\tau)d\tau,$$ if $t\ge r_2-r_1$ and 
  $$	P_+(t,r_2)=P_+(0,r_2-t)+\int_{0}^t(B-r_2j_\theta)(\tau,r_2-t+\tau)d\tau,$$ if $0\le t<r_2-r_1,$ since $t_1(t,r_2)=\max\{0,t-r_2+r_1\}.$
  Thus we obtain the lemma for $P_+(t,r_2)$. This completes the proof.
\end{proof}
\begin{remark}\label{missingboundary remark}
    Lemma \ref{Pplus lemma} allows us to represent the ``unknown" boundary values $B^b(t,r_1)$ and $B^b(t,r_2)$ in terms of only the given initial data \eqref{initial}, the boundary data \eqref{boundary}, and the integrals of $B$ and $j_\theta.$ This is because we have
    $$B^b(t,r_1)=\frac{P_+(t,r_1)}{r_1}-E^b_\theta(t,r_1),$$
    and
    $$B^b(t,r_2)=\frac{P_+(t,r_2)}{r_2}-E_\theta^b(t,r_2).$$
\end{remark}
 Finally the previous lemmas imply the following uniform a priori $L^\infty$ bounds for the fields $E_\theta$ and $B$. 
 \begin{proposition}\label{aprioriestimatesEB}We have
 	\begin{equation}\notag \begin{split}&\|E_\theta\|_{L^\infty ([0,t]\times {[r_1,r_2] })}\le \tilde{C} e^{Ct},\\
 	&\|B\|_{L^\infty ([0,t]\times {[r_1,r_2] })}\leq  \tilde{C} e^{Ct},\end{split}\end{equation}
 where $\tilde{C}$ and $C$ are defined as \eqref{tildeCdef} and \eqref{Cdef} and depend only on $r_1,r_2,$ $\|B^0\|_{L^\infty({[r_1,r_2] })}$, $\|E_\theta^0\|_{L^\infty({[r_1,r_2] })}$, $\|E_\theta^b\|_{L^\infty([0,t]\times\partial{[r_1,r_2] })}$, $\|p^0f^0\|_{L^1({[r_1,r_2] }\times \rtw)},$ $\lambda$, and $t$.
 \end{proposition}
\begin{remark}
It is worthwhile to mention that the $L^\infty$-bounds of the fields $E_\theta$ and $B$ are exponentially growing in time by Proposition \ref{aprioriestimatesEB}. This exponential growth of the fields is the outcome of the appearance of the inhomogeneous source term in the wave equations for $E_\theta$ and $B$ and the Gr\"onwall inequality in the mathematical viewpoint, but this exponential growth is indeed physically relevant in the geometry of the annulus (or the disk) in the physical viewpoint. The reasoning behind this is on the relationship between the fields $E_\theta$ and $B$ via Amp\`{e}re's law $\eqref{pMaxwell}_2$-$\eqref{pMaxwell}_3$; i.e., the curl of each can determine the other. Therefore, the symmetry in the $x_2$-direction assumed in \cite{MR3294216} is very strong as all of the fields must be constant in the $x_2$-direction, while the fields interact via the curl of each other. Therefore, if we just assume the symmetry in the $\theta$-direction as in this paper, the fields can accelerate each other and we have less restrictions than the $x_2$-symmetric situation in the case of the magnetic confinement in an infinite strip \cite[Corollary 2.4]{MR3294216}, where the fields grow linearly in time. Of course, our proposition does not guarantee the minimal growth rates on the fields.

\end{remark} 
\begin{proof}[Proof of Proposition~\ref{aprioriestimatesEB}]
First of all, it follows from Formula~\eqref{e:Char_Formula_for_B} and $P_{\pm}\eqdef rE_\theta\pm rB$ that 
\begin{multline}\notag
    (rB)(t,r)=\frac{1}{2}\bigg(P_+(t_1(t,r),r-t+t_1(t,r))+P_+(t_2(t,r),r+t-t_2(t,r))\\
   -2(r+t-t_2(t,r))E_\theta(t_2(t,r),r+t-t_2(t,r)) \bigg)\\+\frac{1}{2}\int_{t_1(t,r)}^t(B-rj_\theta)(\tau,r-t+\tau)d\tau-\frac{1}{2}\int_{t_2(t,r)}^t\left[(B-rj_\theta)(\tau,r+t-\tau)\right]d\tau,
\end{multline}where $\Blue{t_1(t,r)}=\max\{0,t-r+r_1\}$ and $\Blue{t_2(t,r)}=\max\{0,t-r_2+r\}, $ since
\begin{multline*}-P_-(t_2(t,r),r+t-t_2(t,r))\\=P_+(t_2(t,r),r+t-t_2(t,r))
   -2(r+t-t_2(t,r))E_\theta(t_2(t,r),r+t-t_2(t,r)).\end{multline*}
   Then using $ r+t-\Blue{t_2(t,r)}\le r_2$, we have
 	\begin{multline}\label{prop2.6}|B(t,r)|\leq \frac{1}{2r_1}\bigg(|P_+(t_1(t,r),r-t+t_1(t,r))+P_+(t_2(t,r),r+t-t_2(t,r))|\\
   +2r_2|E_\theta(t_2(t,r),r+t-t_2(t,r))| \\
 	+\int_{t_1(t,r)}^t |B(\tau,r-t+\tau)|d\tau
 	+\int_{t_1(t,r)}^t(r|j_\theta|)(\tau,r-t+\tau)d\tau\\ \int_{t_2(t,r)}^t |B(\tau,r+t-\tau)|d\tau+\int_{t_2(t,r)}^t(r|j_\theta|)(\tau,r+t-\tau)d\tau\bigg).
 	\end{multline} Here we note that both tuples $(t_1(t,r),r-t+t_1(t,r))$ and $(t_2(t,r),r-t+t_2(t,r))$ are either on the initial line $t=0$ or on the boundaries $r=r_1$ or $r=r_2.$ Thus, note that $|E_\theta(t_2(t,r),r+t-t_2(t,r))|$ is given by either $E^0_\theta$ of \eqref{initial} or $E^b_\theta$ of \eqref{boundary}. We can also express $P_+$ using Lemma \ref{Pplus lemma}. Now define $$M_1\eqdef \left\lfloor \frac{\Blue{t_1(t,r)}}{r_2-r_1}\right\rfloor\text{ and } M_2\eqdef \left\lfloor \frac{\Blue{t_2(t,r)}}{r_2-r_1}\right\rfloor.$$ By Lemma \ref{Pplus lemma}, 
 	we have,
 	for even $M_1$,
     \begin{multline}\label{P+estimates1}
        |P_+(t_1(t,r),r-t+t_1(t,r))|\\
        \le r_2 \bigg(\|E^0_\theta\|_{L^\infty([r_1,r_2])}+\|B^0\|_{L^\infty([r_1,r_2])}\bigg) 
 	 	+4r_2(M_1+1)\|E^b_\theta\|_{L^\infty([0,\Blue{t_1(t,r)}]\times\partial[r_1,r_2])}\\+\int_{0}^{\Blue{t_1(t,r)}-M_1(r_2-r_1)}(|B|+r_1|j_\theta|)(\tau,r_1+\Blue{t_1(t,r)}-M_1(r_2-r_1)-\tau)d\tau\\+1_{M_1\ge 2}\sum_{k=1}^{M_1/2}\int_{\Blue{t_1(t,r)}-2k(r_2-r_1)}^{\Blue{t_1(t,r)}-(2k-1)(r_2-r_1)}(|B|+r_2|j_\theta|)(\tau,r_2-\Blue{t_1(t,r)}+(2k-1)(r_2-r_1)+\tau)d\tau\\
        +1_{M_1\ge 2}\sum_{k=0}^{M_1/2-1}\int^{\Blue{t_1(t,r)}-2k(r_2-r_1)}_{\Blue{t_1(t,r)}-(2k+1)(r_2-r_1)}(|B|+r_1|j_\theta|)(\tau,r_1+\Blue{t_1(t,r)}-2k(r_2-r_1)-\tau)d\tau,
    \end{multline} and for odd $M_1$,
     \begin{multline}\label{P+estimates2}
        |P_+(t_1(t,r),r-t+t_1(t,r))|\\
        \le r_2 \bigg(\|E^0_\theta\|_{L^\infty([r_1,r_2])}+\|B^0\|_{L^\infty([r_1,r_2])}\bigg) 
 	 	+4r_2(M_1+1)\|E^b_\theta\|_{L^\infty([0,\Blue{t_1(t,r)}]\times\partial[r_1,r_2])}\\+\int_{0}^{\Blue{t_1(t,r)}-M_1(r_2-r_1)}(|B|+r_2|j_\theta|)(\tau,r_2-\Blue{t_1(t,r)}+M_1(r_2-r_1)+\tau)d\tau\\+1_{M_1\ge 3}\sum_{k=1}^{\frac{M_1-1}{2}}\int_{\Blue{t_1(t,r)}-2k(r_2-r_1)}^{\Blue{t_1(t,r)}-(2k-1)(r_2-r_1)}(|B|+r_2|j_\theta|)(\tau,r_2-\Blue{t_1(t,r)}+(2k-1)(r_2-r_1)+\tau)d\tau\\
        +\sum_{k=0}^{\frac{M_1-1}{2}}\int^{\Blue{t_1(t,r)}-2k(r_2-r_1)}_{\Blue{t_1(t,r)}-(2k+1)(r_2-r_1)}(|B|+r_1|j_\theta|)(\tau,r_1+\Blue{t_1(t,r)}-2k(r_2-r_1)-\tau)d\tau.
    \end{multline}
 	Similarly, by Lemma \ref{Pplus lemma}, 
 	we have,
 	for even $M_2$,
     \begin{multline}\label{P+estimates3}
        |P_+(t_2(t,r),r+t-t_2(t,r))|\\
        \le r_2 \bigg(\|E^0_\theta\|_{L^\infty([r_1,r_2])}+\|B^0\|_{L^\infty([r_1,r_2])}\bigg) 
 	 	+4r_2(M_2+1)\|E^b_\theta\|_{L^\infty([0,\Blue{t_2(t,r)}]\times\partial[r_1,r_2])}\\
 	 	+1_{M_2\ge 2}\int_{\Blue{t_2(t,r)}-(r_2-r_1)}^{\Blue{t_2(t,r)}} (|B|+r_2|j_\theta|)(\tau,r_2-\Blue{t_2(t,r)}+\tau)d\tau\\+\int_{0}^{\Blue{t_2(t,r)}-M_2(r_2-r_1)}(|B|+r_2|j_\theta|)(\tau,r_2-\Blue{t_2(t,r)}+M_2(r_2-r_1)+\tau)d\tau\\+1_{M_2\ge 4}\sum_{k=1}^{\frac{M_2-2}{2}}\int_{\Blue{t_2(t,r)}-2k(r_2-r_1)}^{\Blue{t_2(t,r)}-(2k-1)(r_2-r_1)}(|B|+r_2|j_\theta|)(\tau,r_2-\Blue{t_2(t,r)}+(2k-1)(r_2-r_1)+\tau)d\tau\\
        +1_{M_2\ge 2}\sum_{k=0}^{\frac{M_2-2}{2}}\int^{\Blue{t_2(t,r)}-2k(r_2-r_1)}_{\Blue{t_2(t,r)}-(2k+1)(r_2-r_1)}(|B|+r_1|j_\theta|)(\tau,r_1+\Blue{t_2(t,r)}-2k(r_2-r_1)-\tau)d\tau,
    \end{multline} and for odd $M_2$,
     \begin{multline}\label{P+estimates4}
        |P_+(t_2(t,r),r+t-t_2(t,r))|\\
        \le r_2 \bigg(\|E^0_\theta\|_{L^\infty([r_1,r_2])}+\|B^0\|_{L^\infty([r_1,r_2])}\bigg) 
 	 	+4r_2(M_2+1)\|E^b_\theta\|_{L^\infty([0,\Blue{t_2(t,r)}]\times\partial[r_1,r_2])}\\
 	 	+ \int_{\Blue{t_2(t,r)}-(r_2-r_1)}^{\Blue{t_2(t,r)}} (|B|+r_2|j_\theta|)(\tau,r_2-\Blue{t_2(t,r)}+\tau)d\tau\\+\int_{0}^{\Blue{t_2(t,r)}-M_2(r_2-r_1)}(|B|+r_1|j_\theta|)(\tau,r_1+\Blue{t_2(t,r)}-M_2(r_2-r_1)-\tau)d\tau\\+1_{M_2\ge 3}\sum_{k=1}^{\frac{M_2-1}{2}}\int_{\Blue{t_2(t,r)}-2k(r_2-r_1)}^{\Blue{t_2(t,r)}-(2k-1)(r_2-r_1)}(|B|+r_2|j_\theta|)(\tau,r_2-\Blue{t_2(t,r)}+(2k-1)(r_2-r_1)+\tau)d\tau\\
        +1_{M_2\ge 3}\sum_{k=0}^{\frac{M_2-3}{2}}\int^{\Blue{t_2(t,r)}-2k(r_2-r_1)}_{\Blue{t_2(t,r)}-(2k+1)(r_2-r_1)}(|B|+r_1|j_\theta|)(\tau,r_1+\Blue{t_2(t,r)}-2k(r_2-r_1)-\tau)d\tau.
    \end{multline}
Then we plug \eqref{P+estimates1}-\eqref{P+estimates4} into \eqref{prop2.6} and apply Lemma \ref{jtheta} with Lemma \ref{energydissipation} and Proposition \ref{E1estimate} to obtain
 		\begin{multline}\label{2.23}
 		\|B(t)\|_{L^\infty({[r_1,r_2] })}\\
 		\leq \frac{r_2}{2r_1}\bigg( 4\|E_\theta^0\|_{L^\infty({[r_1,r_2] })}+2\|B^0\|_{L^\infty({[r_1,r_2] })}+\left(8\left\lceil\frac{t}{r_2-r_1}\right\rceil+2\right)\|E_\theta^b\|_{L^\infty([0,t]\times\partial{[r_1,r_2] })}\bigg)\\+\frac{1}{2r_1}\bigg(2\int_0^t \|B(\tau)\|_{L^\infty({[r_1,r_2] })}d\tau\\+2r_2^2\left\lceil\frac{t}{r_2-r_1}\right\rceil\left((\|f^0\|_{L^1({[r_1,r_2] }\times \rtw)}+\lambda)^2+\|E^0_\theta\|^2_{L^\infty}+\|B^0\|^2_{L^\infty}\right)\\
 	+2\left\lceil\frac{t}{r_2-r_1}\right\rceil\|rp^0f^0\|_{L^1({[r_1,r_2] }\times \rtw)}\\+4r_2\left\lceil\frac{t}{r_2-r_1}\right\rceil\int_0^t\left(|(E^b_\theta B^b)(\tau,r_2)|+|(E^b_\theta B^b)(\tau,r_1)|\right)d\tau\bigg).
 \end{multline}
\Blue{For further explanations on how we obtain \eqref{2.23}, see Remark \ref{remark2.9} below.}  Hence, by \eqref{2.23} we obtain
 \begin{multline}\label{Bestimate}\|B(t)\|_{L^\infty({[r_1,r_2] })}\\
 		\leq  \frac{r_2}{r_1}\bigg( 2\|E_\theta^0\|_{L^\infty({[r_1,r_2] })}+\|B^0\|_{L^\infty({[r_1,r_2] })}+\left(4\left\lceil\frac{t}{r_2-r_1}\right\rceil+2\right)\|E_\theta^b\|_{L^\infty([0,t]\times\partial{[r_1,r_2] })}\bigg)\\+\frac{1}{2r_1}\bigg(2r_2^2\left\lceil\frac{t}{r_2-r_1}\right\rceil\left((\|f^0\|_{L^1({[r_1,r_2] }\times \rtw)}+\lambda)^2+\|E^0_\theta\|^2_{L^\infty}+\|B^0\|^2_{L^\infty}\right)\\
 	+2\left\lceil\frac{t}{r_2-r_1}\right\rceil\|rp^0f^0\|_{L^1({[r_1,r_2] }\times \rtw)}\bigg)\\+\frac{1}{r_1}\left(1+4r_2\left\lceil\frac{t}{r_2-r_1}\right\rceil\|E^b_\theta\|_{L^\infty([0,t]\times\partial[r_1,r_2])}\right)\int_0^t \|B(\tau)\|_{L^\infty({[r_1,r_2] })}d\tau\\
 		\eqdef \tilde{C}+C\int_0^t \|B(\tau)\|_{L^\infty({[r_1,r_2] })}d\tau,
 \end{multline}
 		where $\tilde{C}$ and $C$ are defined as
 		\begin{multline}\label{tildeCdef}
 		    \tilde{C}\eqdef  \frac{r_2}{r_1}\bigg( 2\|E_\theta^0\|_{L^\infty({[r_1,r_2] })}+\|B^0\|_{L^\infty({[r_1,r_2] })}\\+\left(4\left\lceil\frac{t}{r_2-r_1}\right\rceil+2\right)\|E_\theta^b\|_{L^\infty([0,t]\times\partial{[r_1,r_2] })}\bigg)\\+\frac{1}{2r_1}\bigg(2r_2^2\left\lceil\frac{t}{r_2-r_1}\right\rceil\left((\|f^0\|_{L^1({[r_1,r_2] }\times \rtw)}+\lambda)^2+\|E^0_\theta\|^2_{L^\infty}+\|B^0\|^2_{L^\infty}\right)\\
 	+2\left\lceil\frac{t}{r_2-r_1}\right\rceil\|rp^0f^0\|_{L^1({[r_1,r_2] }\times \rtw)}\bigg),
 		\end{multline}
 		and
 			\begin{equation}\label{Cdef}
 		    C\eqdef \frac{1}{r_1}\left(1+4r_2\left\lceil\frac{t}{r_2-r_1}\right\rceil\|E^b_\theta\|_{L^\infty([0,t]\times\partial[r_1,r_2])}\right).
 		\end{equation}
 		Note that 
 		$\tilde{C}$ and $C$ are functions depending only on $r_1,r_2,$ and given data  $\|B^0\|_{L^\infty({[r_1,r_2] })}$, $\|E_\theta^0\|_{L^\infty({[r_1,r_2] })}$, $\|E_\theta^b\|_{L^\infty([0,t]\times\partial{[r_1,r_2] })}$, $\|rp^0f^0\|_{L^1({[r_1,r_2] }\times \rtw)},$ $\lambda$, and $t$.
Then by the Gr\"{o}nwall lemma, we obtain 
 	\begin{equation}\label{BboundLinfty}\|B(t)\|_{L^\infty({[r_1,r_2] })}\le \tilde{C}e^{C t}.\end{equation}
 	
 	For the estimate on $E_\theta$, we can directly apply the same argument as in the estimation on $B$, since the right-hand sides of \eqref{e:Char_Formula_for_E_theta} and \eqref{e:Char_Formula_for_B} are essentially the same except for changes of some positive and negative signs. Thus, by using the estimate \eqref{BboundLinfty}, we also have
 	\begin{multline}\label{EboundLinfty}\|E_\theta(t)\|_{L^\infty({[r_1,r_2] })}
 	 \leq \tilde{C}+C\int_0^t \|B(\tau)\|_{L^\infty({[r_1,r_2] })}d\tau\\
 	 \le \tilde{C}+ \tilde{C}(e^{Ct}-1)=\tilde{C}e^{Ct}.
 	 \end{multline}
 	This completes the proof of Proposition \ref{aprioriestimatesEB}.
 \end{proof}
 
\begin{remark}\label{remark2.9}
 In this remark, we briefly explain how we obtain the first bound on $B(t)$ in \eqref{2.23}. We briefly explain how we get the exact constants for the upper-bound of $\|B(t)\|_{L^\infty([r_1,r_2])}$. By \eqref{prop2.6}, the contributions on the upper-bound for $\|B(t)\|_{L^\infty([r_1,r_2])}$ are the followings: the bounds for $|P_+(t_i,\cdot)|$, $|E_\theta|$, $\int_{t_1}^t (B+r|j_\theta|) d\tau$, and $\int_{t_2}^t (B+r|j_\theta|) d\tau$. We note that $M_1$ and $M_2$ can be either even or odd, and here we just introduce the case that both are even. Other cases are similar. 

For the contributions on $P_+(t_1,\cdot) $ and $P_+(t_2,\cdot)$ we use \eqref{P+estimates1} and \eqref{P+estimates3} since $M_1$ and $M_2$ are even. Here we note that there appear in the upper-bound $r_2$ copies of $\|E^0_\theta\|_{L^\infty}$ and $\|B^0_\theta\|_{L^\infty}$ and $4r_2(M_1+1)$ copies of $\|E^b_\theta\|_{L^\infty}$ in the upper-bound. Then together with the $2r_2$ copies of $E_\theta $ in \eqref{prop2.6} which is either $E^0_\theta$ or $E^b_\theta$, we obtain the upper bounds of \begin{multline*}
   \qquad\qquad \frac{r_2}{2r_1}\bigg( 4\|E_\theta^0\|_{L^\infty({[r_1,r_2] })}+2\|B^0\|_{L^\infty({[r_1,r_2] })}\\+\left(8\left\lceil\frac{t}{r_2-r_1}\right\rceil+2\right)\|E_\theta^b\|_{L^\infty([0,t]\times\partial{[r_1,r_2] })}\bigg)
\end{multline*} in \eqref{2.23}.
The leftovers in the contributions of $|P_+(t_1,\cdot)| $ and $|P_+(t_2,\cdot)|$ via \eqref{P+estimates1} and \eqref{P+estimates3} are the integrals on $B$ and $r|j_\theta|$ in the different time intervals. If we consider $\|B\|_{L^\infty([r_1,r_2])}$ in the integral, then we can patch all the time intervals $\{[0,t_1-M_1(r_2-r_1)]$, $[t_1-M_1(r_2-r_1),t_1-(M_1-1)(r_2-r_1)]$, $...$, and $ [t_1-(r_2-r_1),t_1]\}$  and also $\{[0,t_2-M_2(r_2-r_1)]$, $[t_2-M_2(r_2-r_1),t_2-(M_2-1)(r_2-r_1)]$, $...$, and $ [t_2-(r_2-r_1),t_2]\}$ and obtain $\int_0^{t_1} \|B\|_{L^\infty}d\tau$ and $\int_0^{t_2} \|B\|_{L^\infty}d\tau$. Then together with the upper-bounds of $\int_{t_1}^t |B(\tau,r-t+\tau)| d\tau$ and $\int_{t_2}^t |B(\tau,r+t-\tau)| d\tau$ appearing in \eqref{prop2.6}, we obtain exactly two copies of $\int_0^t \|B\|_{L^\infty}d\tau$ in the final upper-bound. So the only thing left in the upper-bound estimate for $B(t)$ is the upper-bounds for \begin{equation}\notag\begin{split}&\int_{t_1-2k(r_2-r_1)}^{t_1-(2k-1)(r_2-r_1)} r_2|j_\theta|(\tau,r_2-t_1+(2k-1)(r_2-r_1)+ \tau)d\tau ,\\ &\int_{t_1-(2k+1)(r_2-r_1)}^{t_1-2k(r_2-r_1)} r_1|j_\theta|(\tau,r_1+t_1-2k(r_2-r_1)- \tau)d\tau ,\\ &\int_{t_2-2k(r_2-r_1)}^{t_2-(2k-1)(r_2-r_1)} r_2|j_\theta|(\tau,r_2-t_2+(2k-1)(r_2-r_1)+ \tau)d\tau , \text{ and}\\  &\int_{t_2-(2k+1)(r_2-r_1)}^{t_2-2k(r_2-r_1)} r_1|j_\theta|(\tau,r_2-t_1-2k(r_2-r_1)- \tau)d\tau \end{split}\end{equation} from \eqref{P+estimates1} and \eqref{P+estimates3}.  For each of the integral, we use the estimate either \eqref{2.10} or \eqref{2.11} in Lemma \ref{jtheta} with different $t's$ and $r's$; for instance,  we choose $t=t_1-(2k-1)(r_2-r_1)$ and $r=r_2$ for the estimate of  $\int_{t_1-2k(r_2-r_1)}^{t_1-(2k-1)(r_2-r_1)} r_2|j_\theta|(\tau,r_2-t_1+(2k-1)(r_2-r_1)+ \tau)d\tau$ such that $t_1(t,r)=t_1-2k(r_2-r_1)$ in \eqref{2.11} of Lemma \ref{jtheta}.   
Then for each piece of the temporal integral, we will have one copy of $\int r'e(\cdot,\cdot)dr'$ bound and $\int r_2 |E^b_\theta B^b|(\tau,\cdot )d\tau$ in the upper-bound by Lemma \ref{jtheta}, which will further be bounded from above by Lemma \ref{energydissipation}. This corresponds to the rest of the upper-bound \begin{multline*}\qquad\qquad\frac{1}{2r_1}\bigg(2r_2^2\left\lceil\frac{t}{r_2-r_1}\right\rceil\left((\|f^0\|_{L^1({[r_1,r_2] }\times \rtw)}+\lambda)^2+\|E^0_\theta\|^2_{L^\infty}+\|B^0\|^2_{L^\infty}\right)\\
 	+2\left\lceil\frac{t}{r_2-r_1}\right\rceil\|rp^0f^0\|_{L^1({[r_1,r_2] }\times \rtw)}\\+4r_2\left\lceil\frac{t}{r_2-r_1}\right\rceil\int_0^t\left(|(E^b_\theta B^b)(\tau,r_2)|+|(E^b_\theta B^b)(\tau,r_1)|\right)d\tau\bigg) \end{multline*}in \eqref{2.23}. 
 \end{remark}
 
 \section{Confinement of the plasma for all time}\label{confinement}
 
 This section is devoted to proving the magnetic confinement of the plasma in the spatial domain. For any initial point $(r,p_r,p_\theta)\in [r_1+\delta_0,r_2-\delta_0]\times \rtw,$ we first define the characteristics $R(s)=R(s;0,r,p_r,p_\theta)$, which initially started in a compactly supported set, will never reach the boundary. Fixing $(r,p_r,p_\theta)\in  [r_1+\delta_0,r_2-\delta_0]\times \rtw$, we define the characteristics for the system \eqref{pVlasov} corresponding to the initial point $(r,p_r,p_\theta)$ as the solution $$s\mapsto (R(s),P_r(s),P_\theta(s))=(R(s;0,r,p_r,p_\theta),P_r(s;0,r,p_r,p_\theta),P_\theta(s;0,r,p_r,p_\theta))$$ that solves
 \begin{equation}\label{characteristics}
 \begin{split}
&\frac{dR}{ds}=\hat{P}_r,\\
&\frac{dP_r}{ds}=E_r+\hat{P}_\theta \bar{B}+\frac{P_\theta^2}{RP^0},\\
&\frac{dP_\theta}{ds}=E_\theta-\hat{P}_r\bar{B}-\frac{P_rP_\theta}{RP^0},\\
&R(0;0,r,p_r,p_\theta)=r,\ \ P_r(0;0,r,p_r,p_\theta)=p_r,\ \ P_\theta(0;0,r,p_r,p_\theta)=p_\theta,
\end{split}
\end{equation} where $P^0 \eqdef \sqrt{1+P_r^2+P_\theta^2}$, $\hat{P}_r \eqdef \frac{P_r}{P^0}$ and $\hat{P}_\theta \eqdef \frac{P_\theta}{P^0}$. Furthermore, the functions $E_r$, $E_\theta$ and $\bar{B}$ in \eqref{characteristics} are all evaluated at the point $(s,R(s))$.
We can write the self-consistent magnetic field $B$ in terms of its potential $\psi$; more precisely, we have
\begin{equation}\label{magpotentialeq}
B(t,r)=\frac{1}{r}\left(\frac{\partial(r\psi(t,r))}{\partial r}\right).\end{equation} Without loss of generality, we additionally suppose that $\psi$ satisfies \begin{equation}
    \label{potential assumption at mid}\psi(t,r_m)=0
\end{equation} for all $t\in [0,T]$, where $r_m\eqdef \frac{r_1+r_2}{2}$ is the median radius. \Blue{Otherwise, we can consider our potential $\tilde{\psi}$ as $\tilde{\psi}(t,r)= \psi(t,r)-\frac{r_m}{r}\psi(t,r_m).$ }

Furthermore, if $E_r$, $E_\theta$, $B \in C^1([0,T]\times {[r_1,r_2] })$, then the $C^1$ solutions to the system \eqref{characteristics} exist for a finite time and can be extended to the whole time interval $[0,T]$ if $R(s)$ does not hit the spatial boundary $R(s)=r_1$ or $R(s)=r_2$ for any $s\in[0,T]$. Via Lemma~\ref{magconfi} below, we will prove that the characteristic $R(s)$ never reach the spatial boundary $\partial\Omega$, provided that the external magnetic field $B_{ext}$ is well-chosen. Throughout this section, we will omit the dependency on $\theta$, since we assume that all the functions $f$, $E_r$, $E_\theta,$ and $B$ are rotationally symmetric, namely they are independent of $\theta$.

\subsection{Construction of the external magnetic potential}\label{sec.3.1}In this section, we first construct an infinite time-independent external magnetic potential $\psi_{base}=\psi_{base}(r)$ that confines the charged particles. Later, in Section \ref{sec.truncated}, by defining a \textit{moving bar} that increases in time, we will be able to truncate the infinite potential $\psi_{base}$ and construct a finite time-dependent external magnetic potential that also confines the charged particles.

As mentioned in Remark \ref{catch22}, our \textit{truncation} method in this section makes sense and works because we can prove that the a priori estimates on the self-consistent electromagnetic fields $E_r,$ $E_\theta$, and $B$ are independent of the external magnetic potential $\psi_{ext}$. Hence, we can also obtain the velocity bound \eqref{Vbound} independent of the external magnetic potential as in Lemma \ref{Vbound lemma} and this allows us to choose a finite barrier independent of the external potential, which will be used for the truncation.

To begin with, as we introduced in Hypothesis \ref{psibasegeneral}, the minimal sufficient conditions for the magnetic confinement that we require on the time-independent infinite potential $\psi_{base}$ are as follows; for a given distance $\delta\in (0,\delta_0)$ from the spatial boundary $\partial\Omega$, we assume
\begin{enumerate}
	\item $\psi_{base} \in C^2((r_1+\delta,r_2-\delta)).$
	\item $\psi_{base}$ satisfies
	\[	\lim_{r\to (r_1+\delta)^+} | \psi_{base}(r) | = \lim_{r\to (r_2-\delta)^-} | \psi_{base}(r) | = \infty .	\]
\end{enumerate}
We first show that the infinite external potential $\psi_{base}$ can be used to confine all the charged particles in the interior as in the following lemma:

\begin{lemma}\label{magconfi}
	Assume $E_r$, $E_\theta$, $B \in C^1([0,T]\times {[r_1,r_2] })$ satisfy Maxwell's equations \eqref{pMaxwell}. Suppose  \begin{equation}\label{boundsEB}
	|\vec{E}(t,r)|,\;|B(t,r)|\leq \tilde{C} e^{Ct},
	\end{equation} for any $(t,r)\in [0,T]\times (r_1,r_2) $ where $C$ and $\tilde{C}>0$ are the same functions defined as \eqref{tildeCdef} and \eqref{Cdef} in Proposition~\ref{aprioriestimatesEB}. Fix any $(r,p_r,p_\theta)\in  [r_1+\delta_0,r_2-\delta_0]\times \{p\in\rtw : |p|\leq M_0\}$ for some $\delta_0 \in \left(0,\frac{r_2-r_1}{2}\right)$ and $M_0>0$. Consider the characteristics $(s,R(s),P_r(s),P_\theta(s))\eqdef (s,R(s;0,r,p_r,p_\theta),P_r(s;0,r,p_r,p_\theta),P_\theta(s;0,r,p_r,p_\theta))$ of the system \eqref{pVlasov} corresponding to the point $(0,r,p_r,p_\theta)$ as the solutions of the system of ODEs \eqref{characteristics}. Suppose that the external magnetic field $B_{ext}$ is defined via a given time-independent potential $\psi_{base}(r)$ in Hypothesis \ref{psibasegeneral} as  $$B_{ext}=B_{base}(r)\eqdef \frac{1}{r}\frac{\partial(r\psi_{base}(r))}{\partial r},$$ where $\psi_{base}$ is defined as in Hypothesis \ref{psibasegeneral}. Then we have for any $\theta\in [-\pi,\pi )$,
	\begin{equation*}
	\mathrm{dist}(R(s)\hat{r} + \theta \hat{\theta} ,\partial\Omega) \geq \mathrm{dist}(U_{\delta_0}(s),\partial\Omega)>\delta >0.
	\end{equation*}	 for any $s\in [0,T]$, where $U_{\delta_0}(s)$ is defined as in \eqref{Udelta0}.
	
	Furthermore, if we define $\psi_{base}$ as  \eqref{psibasedef},  then we have for any $\theta\in [-\pi,\pi )$,
	\begin{equation}\label{magconfi2}
	\mathrm{dist}(R(s)\hat{r} + \theta \hat{\theta} ,\partial\Omega)\ge  \left(\frac{r_2-r_1}{\pi}\right)\arcsin(C_s)>\delta >0,
	\end{equation} for any $s\in [0,T]$ where $C_s$ is a positive constant for each fixed $s\in[0,T]$ which also depends on $\delta_0, \ r_1,\ r_2,\ C,\ \tilde{C}(s),$ and $ M_0$. 
\end{lemma}

\begin{remark}
	We remark that \eqref{boundsEB} is not an actual assumption. Indeed, \eqref{boundsEB} is just a direct consequence of Proposition~\ref{E1estimate} and Proposition~\ref{aprioriestimatesEB}.
\end{remark}
Before we prove Lemma \ref{magconfi}, we first introduce an estimate on the bound of the speed of propagation:
\begin{lemma} Assume \eqref{boundsEB}. Denote $P(0)$ as $P(0)=p$. Then we have \label{Vbound lemma}
\begin{equation}\label{Vbound}
\sup_{\tau\in [0,s]}|P(\tau)|\leq 2\tilde{C}(s)\frac{|e^{Cs}-1|}{C}+|p|,
\end{equation} where $C$ and $\tilde{C}(s)$ are the same functions defined as \eqref{tildeCdef} and \eqref{Cdef} in Proposition \ref{aprioriestimatesEB}.
\end{lemma}
\begin{proof}
A direct computation yields
$$\frac{d}{ds}|P|^2=2\left(P_rE_r+P_r\hat{P}_\theta \bar{B}+\frac{P_rP_\theta^2}{RP^0}+P_\theta E_\theta-P_\theta \hat{P}_r\bar{B}-\frac{P_rP_\theta^2}{RP^0}\right)=2(P_rE_r+P_\theta E_\theta).$$ Then the bounds \eqref{boundsEB} further imply that $$|P(s)|^2\leq |p|^2+2\tilde{C}(s)\left|\int_0^s|P(\tau)|e^{C\tau}{\tiny }d\tau\right|
\leq |p|^2+2\tilde{C}(s)\sup_{\tau\in [0,s]}|P(\tau)| \frac{|e^{Cs}-1|}{C}.
$$
Therefore, we obtain the lemma.
\end{proof}

\begin{proof}[Proof of Lemma \ref{magconfi}]\Blue{Recall that we denote $P= (P_r,P_\theta)$ and $p= (p_r,p_\theta)$.}
Without loss of generality, suppose $r\geq r_m\eqdef \frac{r_1+r_2}{2}$; the case for $r< r_m\eqdef \frac{r_1+r_2}{2}$ will be similar, and left for the interested readers. 
\Blue{By \eqref{characteristics} and \eqref{magpotentialeq}}, we observe that
\begin{multline*}
\frac{dP_\theta}{ds}=E_\theta-\hat{P}_r\bar{B}-\frac{P_rP_\theta}{RP^0}=E_\theta-\frac{\hat{P}_r}{R}\frac{\partial(R(\psi+\psi_{base}))}{\partial R}-\frac{P_rP_\theta}{RP^0}\\
=E_\theta+\partial_t\psi-\frac{1}{R}\frac{d(R(\psi+\psi_{base}))}{ds}-\frac{P_rP_\theta}{RP^0}.
\end{multline*}Therefore, we have\begin{multline*}
\frac{d}{ds}\left(RP_\theta+R(\psi+\psi_{base})\right)=R\partial_t\psi+\dot{R}P_\theta+RE_\theta-\frac{P_rP_\theta}{P^0}
=R\partial_t\psi+RE_\theta.
\end{multline*}
Integrating with respect to $s$ over the time interval $[0,s]$, we have
\begin{multline}RP_\theta+R(\psi(s,R)+\psi_{base}(R))=rp_\theta+r(\psi(0,r)+\psi_{base}(r))\\
+\int_0^s\left(R(\tau)\partial_t\psi(\tau,R(\tau))+R(\tau)E_\theta(\tau,R(\tau))\right)d\tau.\end{multline}
\Blue{Also recall \eqref{potential assumption at mid} that we have assumed that $\psi(t,r_m)=0$ for all $t\in[0,T]$ and hence we have $\partial_t\psi(t,r_m)=0$.} Since
\begin{multline*}
R(\tau)\partial_t(\psi(\tau,R(\tau)))=\int_{r_m}^{R(\tau)} r'\partial_tB(\tau,r')dr'=\int_{r_m}^{R(\tau)} -\partial_{r'}(r'E_\theta(\tau,r'))dr'\\=r_mE_\theta(\tau,r_m)-R(\tau)E_\theta(\tau,R(\tau)),
\end{multline*}we have
\begin{multline}\label{Rpsiext difference}RP_\theta+R(\psi(s,R)+\psi_{base}(R))=rp_\theta+r(\psi(0,r)+\psi_{base}(r))\\
+\int_0^s\left(r_mE_\theta \right)(\tau,r_m)d\tau.\end{multline} Indeed, we can easily control the terms $r\psi(0,r)$ and $R\psi(s,R)$ by integrating \eqref{magpotentialeq} and using the hypothesis $\|B(\tau)\|_{L^\infty([r_1,r_2])}\leq \tilde{C}(\tau)e^{C\tau}$ as follows:
\begin{multline*}|R(\tau)\psi(\tau,R(\tau))|\leq \left|\int_{r_m}^{R(\tau)} yB(\tau,y)dy\right| \\ \leq e^{C\tau}\tilde{C}(\tau)\frac{|R(\tau)^2-r_m^2|}{2}\leq e^{C\tau}\tilde{C}(\tau)\frac{(r_2+r_m)(r_2-r_1)}{4},
\end{multline*} 
for both $\tau = 0$ and $\tau =s$ as we have \[
	|R(\tau)^2-r_m^2| = |R(\tau) + r_m|	\cdot |R(\tau) - r_m| \leq (r_2 + r_m) \cdot \min \left\{ |r - r_m| + \tau , \frac{r_2-r_1}{2} \right\},
\]
since $|R(\tau) - r_m| \leq \left\{ |r - r_m| + \tau , \frac{r_2-r_1}{2} \right\}$. 
Using \eqref{Vbound} and $|P_\theta-p_\theta|\leq |P|+|p|$, we also have 
\begin{multline*}|R-r||P_\theta|+ r|P_\theta-p_\theta|\leq (r_2-r_1)|P_\theta|+r_2(|P|+|p|)\\
\le (2r_2-r_1)\left(\frac{2\tilde{C}(s)}{C}(e^{Cs}-1)+|p|\right)+r_2|p|.\end{multline*}
Also, we observe that 
$$\left|\int_0^s\left(r_mE_\theta\right)(\tau,r_m)d\tau\right|\le \frac{\tilde{C}(s) r_m}{C}\left(e^{Cs}-1\right).$$
Thus, we use \eqref{Rpsiext difference} and obtain \begin{multline}\label{Rpsiext diff 3}
|R(s)\psi_{base}(R(s))|\le |r\psi_{base}(r)|+e^{Cs}\tilde{C}(s)\frac{(r_2+r_m)(r_2-r_1)}{2}\\+(2r_2-r_1)\left(\frac{2\tilde{C}(s)}{C}(e^{Cs}-1)+|p|\right) +r_2|p|+\frac{\tilde{C}(s)r_m}{C}(e^{Cs}-1)\\
\le |r\psi_{base}(r)|+Ke^{Cs},
\end{multline}
for any $s\in [0,T]$ for some constant $K=K(C,\tilde{C}(s),r_1,r_2,M_0)>0$ that is defined as
\begin{multline*}
    K(C,\tilde{C}(s),r_1,r_2,M_0)
\\
\eqdef \frac{\tilde{C}(s)}{2}(r_2+r_m)(r_2-r_1)+(2r_2-r_1)\left(\frac{2\tilde{C}(s)}{C}+M_0\right)   +r_2M_0+\frac{\tilde{C}(s)r_m}{C},
\end{multline*}
since $|p|\le M_0$. 
Therefore, for any $s\in [0,T]$,
$$|R(s)\psi_{base}(R(s))|\le |r\psi_{base}(r)|+Ke^{Cs} \le r_2\max_{r\in [r_1+\delta_0,r_2-\delta_0]}|\psi_{base}(r)|+ Ke^{Cs},
$$
and hence,
\begin{equation}\label{e:Upper_Bdd_for_psi(R)}
	|\psi_{base}(R(s))|\le \frac{r_2}{r_1} \left( \max_{r\in [r_1+\delta_0,r_2-\delta_0]}|\psi_{base}(r)| \right) + \frac{K}{r_1} e^{Cs}.
\end{equation}
It follows from Hypothesis \ref{psibasegeneral} that the set
\begin{equation*}
U_{\delta_0}(s) \eqdef \left\{ x \in \Omega; \ |\psi_{base}(x)| \leq \frac{r_2}{r_1} \left( \max_{r\in [r_1+\delta_0,r_2-\delta_0]}|\psi_{base}(r)| \right) + \frac{K}{r_1} e^{Cs} \right\}
\end{equation*}
is a compact and proper subset of the open domain $\Omega$; in addition, we have $\mathrm{dist}(U_{\delta_0}(s),\partial\Omega)>\delta$ by Hypothesis \ref{psibasegeneral}. Thus, the inequality~\eqref{e:Upper_Bdd_for_psi(R)} implies that for any $\theta\in [-\pi,\pi )$,
\begin{equation*}
	\mathrm{dist}(R(s)\hat{r} + \theta \hat{\theta} ,\partial\Omega) \geq \mathrm{dist}(U_{\delta_0}(s),\partial\Omega)>\delta >0.
\end{equation*}

In addition, if we assume that $\psi_{base}$ is in the explicit form of \eqref{psibasedef}, then we further have that 
\begin{multline}\label{Csdef}
\csc \left(\frac{\pi}{r_2-r_1}(R(s)-r_1)\right)\le1+\frac{r}{r_1}\left| \csc \left(\frac{\pi}{r_2-r_1}(r-r_1)\right)-1\right|+ \frac{K}{r_1}e^{Cs}\\ \le 1+\frac{r_2}{r_1}\left| \csc \left(\frac{\pi}{r_2-r_1}(r_2-r_1-\delta_0)\right)-1\right|+ \frac{K}{r_1}e^{Cs}=: C_s,
\end{multline} 
by \eqref{psibasedef} and that the initial distribution is supported only on $[r_1+\delta_0, r_2-\delta_0]$.
Thus we have
$$ r_1+ \left(\frac{r_2-r_1}{\pi}\right)\arcsin(C_s)\le R(s)\le r_2 - \left(\frac{r_2-r_1}{\pi}\right)\arcsin(C_s),$$for all $s\in [0,T]$. This completes the proof for this lemma. 
%
\end{proof} 

Using this lemma, we will construct a \textit{moving bar} $L_{bar}(s)$ that will be used for the truncation of the infinite potential in the next section.

\subsection{Truncated time-dependent external potential}\label{sec.truncated}Equipped with the magnetic confinement via the time-independent external magnetic potential $\psi_{base}=\psi_{base}(r)$ that is infinite at the boundary $r=r_1$ and $r=r_2$, we can now define a moving bar $L_{bar}(s)$ that is an increasing function $L_{bar}:[0,T]\rightarrow [0,\infty)$. The moving bar $L_{bar}(s)$ physically stands for the maximum level of the external potential that we need to impose so that we can confine all the particle trajectory $R(s)$ in the interior domain whose initial state $(R(0),P(0))$ is in the support of $f_0$.  We will use this to truncate the infinite potential $\psi_{base}$ and construct a finite time-dependent external magnetic field that also confines the particles in the interior. The key strategy is to find a moving bar which increases fast enough with respect to time $s\in [0,T]$, so that the truncated potential will still confine all the particles.

For the general form of the external potential $\psi_{base}$ from Hypothesis \ref{psibasegeneral}, we define $L_{bar}(s)$ as 
\begin{equation}\begin{split}
L_{bar}(s)&\eqdef \max_{x\in U_{\delta_0}(s)} |\psi_{base}(x)|,\end{split}
\end{equation}
where $U_{\delta_0}(s)$ is defined as in \eqref{Udelta0}.
Then we define the finite time-dependent external magnetic potential $\psi_{ext}=\psi_{ext}(s,r)$ for $s\in [0,T]$ \Blue{as in \eqref{finitepsiext}. }
Note that as long as $
\psi_{base}(r)\le L_{bar}(s)$, $\psi_{ext}(s,r)$ is equal to $\psi_{base}(r)$, which is time-independent. Therefore, if we consider a particle whose initial state $(R(0),P(0))$ is in the support of $f_0$, then we always have $\psi_{base}(R(s))\le L_{bar}(s)$ for all $s\geq 0$, due to the construction of $L_{bar}$. Therefore, as long as the particle trajectory starts at $(R(0),P(0))$ in the support of $f_0$, we always control the trajectory $(R(s),P(s))$ of the particle via the time-independent potential $\psi_{ext}(s,R(s))= \psi_{base}(R(s))$. Therefore, we can replace $\psi_{base}(R(s))$ by $\psi_{ext}(s,R(s))$ in the proof of Lemma \ref{magconfi}, since $\psi_{ext}(s,R(s))=\psi_{base}(R(s))$ always hold throughout the whole proof.
This completes the proof for Theorem~\ref{maintheorem2}.

The external potential $\psi_{base}$ can be chosen explicitly, such as the form in \eqref{psibasedef};  see Remark \ref{remark.explicit} for more details.
\subsection{A unique global trajectory}
Lemma \ref{magconfi} further implies that any particles which are initially away from the boundary can never reach the spatial boundary. Therefore, we obtain the following corollary on the unique trajectory:
\begin{corollary}\label{uniquetraj}
	Assume $E_r$, $E_\theta$, $B \in C^1([0,T]\times {[r_1,r_2] })$ satisfy Maxwell's equations~\eqref{pMaxwell}. Suppose \begin{equation*}
	|\vec{E}(t,r)|,\; |B(t,r)|\leq \tilde{C} e^{Ct},
	\end{equation*} for all $(t,r)\in [0,T]\times (r_1,r_2)$ where the functions $C$ and $\tilde{C}>0$ are defined as \eqref{tildeCdef} and \eqref{Cdef}. Then for any fixed $(r,p)\in   [r_1+\delta_0,r_2-\delta_0]\times \{|p|\leq M_0\}$ for some $\delta_0 \in \left(0,\frac{r_2-r_1}{2}\right)$ and $M_0>0$, the characteristic ODEs \eqref{characteristics} admits a unique $C^1$ solution $(R(s),P(s))$ in $[0,T]$ with $R(s)\in  (r_1,r_2)$ for all $s\in [0,T]$.
\end{corollary}

\section{Propagation of the \texorpdfstring{$L^\infty$}{} moment and the bounds for the momentum support} \label{Linftysection}
This section is devoted to proving the propagation of the $L^\infty$ moment of the distribution function $f$ and proving that the particle distribution has a compact support in $p$ variable within any finite time interval. In order to prove the global existence and the uniqueness of a classical solution, we need to prove the a priori $L^\infty$ estimate for the distribution solution $f$ and the derivatives of the fields $E$ and $B$ and the distribution $f$. The $L^\infty$ estimate on $f$ will be given in this section, and the estimates on the derivatives will be given in Section \ref{sec:derivatives}. This can be shown as a consequence of the magnetic confinement and the uniqueness of the characteristic trajectory.

We first suppose $(f,E_r,E_\theta,B)$ is a $C^1$ solution to \eqref{pVlasov}. Then another direct consequence of Corollary \ref{uniquetraj} is that the solution $f$ to \eqref{pVlasov} is constant along the unique characteristic trajectory. Therefore, we obtain
\begin{equation}\label{propagationLinfty}
	\|f\|_{L^\infty([0,T]\times {[r_1,r_2] }\times\rtw)}=\|f^0\|_{L^\infty( {[r_1,r_2] }\times\rtw)}.
\end{equation}
In addition, we can also prove that the solution $f(t,r,p_r,p_\theta)$ to the system \eqref{pVlasov} has a compact support in the $p$ variables within any finite time interval.
Define $M(t)$ as follows:
\begin{definition}\label{M.def}Let $(f,E_r,E_\theta,B)$ be a $C^1$ solution to \eqref{pVlasov}. For each $t\in [0,T]$, define the maximum radius of the momentum support $M(t)$ as 
$$M(t)=\max_{p\in \mathrm{supp}_{p}(f)(t)}|p| ,$$ where $$\mathrm{supp}_{p}(f)(t)\eqdef \{p\in \mathbb{R}^2\  |\  f(t,r,p_r,p_\theta)\neq 0,  \text{ for some } r\in[r_1,r_2]\}.$$
\end{definition} Then we have the following estimate:

\begin{lemma}\label{momsupp}
Suppose that $\mathrm{supp} (f^0)\in [r_1+\delta_0,r_2-\delta_0]\times \{|p|\leq M_0\}$ for some $\delta_0\in \left(0,\frac{r_2-r_1}{2}\right)$ and $M_0>0$. Then we have 
$$ M(t)\leq M_0+4\tilde{C}\frac{e^{Ct}}{C},$$
where $C$ and $\tilde{C}$ are the same functions defined as \eqref{tildeCdef} and \eqref{Cdef} obtained in Proposition \ref{aprioriestimatesEB}.
\end{lemma}
\begin{proof}

By Lemma \ref{Vbound lemma}, it follows from a direct computation that $$\frac{d}{ds}|P(s)|^2=2(P_rE_r+P_\theta E_\theta),$$ 
so using Proposition~\ref{E1estimate}, Proposition~\ref{aprioriestimatesEB} and \eqref{Vbound}, we have 
$$|P(t)|^2\leq |P(0)|^2+2\tilde{C}\int_0^t |P(\tau)|e^{C\tau}d\tau\leq |P(0)|^2+2\tilde{C} (2\tilde{C}\frac{e^{Ct}}{C}+|p|)\frac{e^{Ct}}{C}.$$ Thus, $$|P(t)|\leq |P(0)|+4\tilde{C}\frac{e^{Ct}}{C}.$$Since we have assumed that $\mathrm{supp} (f^0)\in [r_1+\delta_0,r_2-\delta_0]\times \{|p|\leq M_0\}$ for some $\delta_0,M_0>0$, we have 
$$M(t) \leq M_0+4\tilde{C}\frac{e^{Ct}}{C}.$$
This completes the proof.
\end{proof}

\subsection{\texorpdfstring{$L^\infty$}{} bounds for the density and the flow} Finally, it is worthwhile to mention that the finite momentum support of the particle distribution implies the following $L^\infty$ bounds on the charge and current densities as well:
\begin{corollary}\label{rhojbound}
	We have \begin{multline*}\|\rho\|_{L^\infty([0,T]\times{[r_1,r_2] })},\ \|j\|_{L^\infty([0,T]\times{[r_1,r_2] })}\\
	\leq \pi \|f^0\|_{L^\infty({[r_1,r_2] }\times \rtw)}\left(M_0+4\tilde{C}\frac{e^{Ct}}{C}\right)^2.\end{multline*}
\end{corollary}\begin{proof}
Note that 
\begin{multline*}\rho(t,x)=\int_\rtw f(t,x,p)dp=\int_{|p|\leq M_0+4\tilde{C}\frac{e^{Ct}}{C}}f(t,x,p)dp\\\leq\pi\|f^0\|_{L^\infty({[r_1,r_2] }\times \rtw)}\left(M_0+4\tilde{C}\frac{e^{Ct}}{C}\right)^2.\end{multline*}Similarly, we have
\begin{multline*}|j(t,x)|\leq\int_\rtw |\hat{p}f(t,x,p)|dp\leq \int_{|p|\leq M_0+4\tilde{C}\frac{e^{Ct}}{C}}|f(t,x,p)|dp\\
\leq\pi\|f^0\|_{L^\infty({[r_1,r_2] }\times \rtw)}\left(M_0+4\tilde{C}\frac{e^{Ct}}{C}\right)^2.\end{multline*}
This completes the proof.
\end{proof}

\section{Estimates for the derivatives}\label{sec:derivatives}
This section is devoted to a priori $L^\infty$ estimates on the derivatives of the fields $(E_r,E_\theta,B)$ and the distribution $f$.
\subsection{Derivatives of \texorpdfstring{$E$}{} and \texorpdfstring{$B$}{}}
We begin with the field $E_r$. Since $\frac{1}{r}\partial_r(rE_r)=\rho,$ we have
$$\|\partial_r (rE_r)\|_{L^\infty([0,T]\times {[r_1,r_2] })}\leq  \pi r_2\|f^0\|_{L^\infty({[r_1,r_2] }\times \rtw)}\left(M_0+4\tilde{C}\frac{e^{Ct}}{C}\right)^2,$$ by Corollary \ref{rhojbound}.

Now recall that in Section~\ref{estiEthetaB} we have defined $P_\pm=rE_\theta\pm rB.$ 
 The rest of this section is devoted to showing the estimates on the derivatives.
Recall that, by Lemma \ref{momsupp}, $f$ has a compact support in $p$ if $f_0$ does. Then, for the confined solutions, we have the following lemma:
 \begin{lemma}Suppose that $f^0$ is supported in $[r_1+\delta_0,r_2-\delta_0]\times \{|p|\leq M_0\}$ for some $\delta_0>0$ and $M_0>0$. Then for the confined solution, we have $$\|\partial_r (rE_\theta)\|_{L^\infty([0,T]\times{[r_1,r_2] })},\|\partial_r (rB)\|_{L^\infty([0,T]\times{[r_1,r_2] })}\leq C_T,$$ for some constant $C_T>0$ which depends only on $M_0$, $T$, $\lambda$, $\|f^0\|_{L^\infty}$, the $C^1$ norm of $E_\theta^0$, $B^0$ on ${[r_1,r_2] }$, and the $C^1$ norm of $E^b_\theta$, $B^b$ on $[0,T]$.
 \end{lemma}

\begin{proof}By \eqref{P+}, \eqref{P-}, 
and the Leibniz rule, the derivative $\partial_r P_+$ is now equal to
	\begin{multline}\label{partialrP+}\partial_r P_+(t,r)
=	Q(t,r)+\int_{t_1(t,r)}^t \partial_rB(\tau,r-t+\tau)d\tau\\-\int_{t_1(t,r)}^t\int_\rtw \hat{p}_\theta f(\tau,r-t+\tau,p_r,p_\theta)dp_rdp_\theta d\tau\\-\int_{t_1(t,r)}^tr\int_\rtw \hat{p}_\theta \partial_rf(\tau,r-t+\tau,p_r,p_\theta)dp_rdp_\theta d\tau,
	\end{multline}where
	\begin{multline*}
	Q(t,r)\eqdef \partial_r (P_+(t_1(t,r),r-t+t_1(t,r))) +\mathbbm{1}_{r<t+r_1}(B-rj_\theta)(t-r+r_1,r_1)\\
	=-\mathbbm{1}_{r<t+r_1}\left(r_1 \partial_t E^b_\theta  (t-r+r_1,	r_1)+r_1 \partial_t B^b  (t-r+r_1,	r_1)-B^b(t-r+r_1,r_1)\right)\\
	+\mathbbm{1}_{r\ge t+r_1} \left((E^0_\theta+B^0)(r-t)+(r-t)\partial_r (E^0_\theta+B^0)(r-t)\right)
	\end{multline*} \Blue{and we recall that $j=0$ at the boundary due to the confinement. }
	Then we have \begin{equation}\label{Q}
\|Q\|_{L^\infty([0,T]\times {[r_1,r_2] })}\leq C_T,
	\end{equation} where $C_T$ is a constant depending on $M_0$, $T$, $\lambda$, $\|f^0\|_{L^\infty}$, the $C^1$ norm of $E_\theta^0$, $B^0$ on ${[r_1,r_2] }$, and the $C^1$ norm of $E^b_\theta$, $B^b$ on $[0,T]$.
	
By treating the radial derivative of $B$ on the right-hand side of \eqref{partialrP+} via considering Amp\`{e}re's circuital law $\eqref{pMaxwell}_3$, 
we obtain\begin{multline*}\partial_r P_+(t,r)
	=	Q(t,r)-\int_{t_1(t,r)}^t \partial_t E_\theta (\tau,r-t+\tau)d\tau\\-2\int_{t_1(t,r)}^t\int_\rtw \hat{p}_\theta f(\tau,r-t+\tau,p_r,p_\theta)dp_rdp_\theta d\tau\\-\int_{t_1(t,r)}^tr\int_\rtw \hat{p}_\theta \partial_rf(\tau,r-t+\tau,p_r,p_\theta)dp_rdp_\theta d\tau.
	\end{multline*}
	Now we are going to use the following splitting of the operator $\partial_r$ motivated by \cite{MR1066384}:
\begin{equation}\label{splittingop}\partial_r = \frac{T_+-S}{1-\hat{p}_r},\end{equation} where $$T_+\eqdef \partial_t +\partial_r\ \text{and}\ S\eqdef \partial_t +\hat{p}_r \partial_r.$$
	Then we further have
	\begin{multline*}\partial_r P_+(t,r)
	=Q(t,r)-\int_{t_1(t,r)}^t \partial_t E_\theta (\tau,r-t+\tau)d\tau\\-2\int_{t_1(t,r)}^t\int_\rtw \hat{p}_\theta f(\tau,r-t+\tau,p_r,p_\theta)dp_rdp_\theta d\tau\\-\int_{t_1(t,r)}^t\int_\rtw\frac{d}{d\tau} \frac{r\hat{p}_\theta}{1-\hat{p}_r}f(\tau,r-t+\tau,p_r,p_\theta)dp_rdp_\theta d\tau\\+\int_{t_1(t,r)}^t\int_\rtw \frac{ r\hat{p}_\theta}{1-\hat{p}_r}Sf(\tau,r-t+\tau,p_r,p_\theta)dp_rdp_\theta d\tau.
	\end{multline*}
Now we implement the temporal integration and use the Vlasov equation $$Sf+\left(E_r+\hat{p}_\theta\bar{B}+\frac{p_\theta^2}{rp^0}\right)\partial_{p_r}f+\left(E_\theta-\hat{p}_r\bar{B}-\frac{p_rp_\theta}{rp^0}\right)\partial_{p_\theta} f=0$$ to obtain that 
	\begin{multline*}\partial_r P_+(t,r)
=Q(t,r)- E_\theta (t,r)+E_\theta(t_1,r-t+t_1)\\-2\int_{t_1(t,r)}^t\int_\rtw \hat{p}_\theta f(\tau,r-t+\tau,p_r,p_\theta)dp_rdp_\theta d\tau\\-\int_\rtw  \frac{r\hat{p}_\theta}{1-\hat{p}_r}f(t,r,p_r,p_\theta)dp_rdp_\theta+\int_\rtw  \frac{r\hat{p}_\theta}{1-\hat{p}_r}f(t_1,r-t+t_1,p_r,p_\theta)dp_rdp_\theta\\-\int_{t_1(t,r)}^t\int_\rtw \frac{ r\hat{p}_\theta}{1-\hat{p}_r}\bigg(\left(E_r+\hat{p}_\theta\bar{B}+\frac{p_\theta^2}{rp^0}\right)\partial_{p_r}f\\+\left(E_\theta-\hat{p}_r\bar{B}-\frac{p_rp_\theta}{rp^0}\right)\partial_{p_\theta} f\bigg)(\tau,r-t+\tau,p_r,p_\theta)dp_rdp_\theta d\tau.
\end{multline*} 
Thus we have
\begin{multline*}\partial_r P_+(t,r)
=Q(t,r)- E_\theta (t,r)+E_\theta(t_1(t,r),r-t+t_1(t,r))\\-2\int_{t_1(t,r)}^t\int_\rtw \hat{p}_\theta f(\tau,r-t+\tau,p_r,p_\theta)dp_rdp_\theta d\tau\\-\int_\rtw  \frac{r\hat{p}_\theta}{1-\hat{p}_r}f(t,r,p_r,p_\theta)dp_rdp_\theta+\int_\rtw  \frac{r\hat{p}_\theta}{1-\hat{p}_r}f(t_1,r-t+t_1,p_r,p_\theta)dp_rdp_\theta\\+\int_{t_1(t,r)}^t\int_\rtw \bigg[\partial_{p_r}\left(\left(\frac{ r\hat{p}_\theta}{1-\hat{p}_r}\right)\left(E_r+\hat{p}_\theta\bar{B}+\frac{p_\theta^2}{rp^0}\right)\right)f\\+\partial_{p_\theta}\left(\left(\frac{r\hat{p}_\theta}{1-\hat{p}_r}\right)\left(E_\theta-\hat{p}_r\bar{B}-\frac{p_rp_\theta}{rp^0}\right)\right) f\bigg](\tau,r-t+\tau,p_r,p_\theta)dp_rdp_\theta d\tau,
\end{multline*} 
by an integration by parts. \Blue{Here we note that the denominators $1-\hat{p}_r$ are bounded below as due to the compact momentum support by Lemma \ref{momsupp} as $1-\hat{p}_r>c(T,M_0,C,\tilde{C})>0.$} Therefore, we use Proposition \ref{aprioriestimatesEB}, Lemma \ref{momsupp}, \eqref{Q}, and \eqref{propagationLinfty} to obtain that 
$$\|\partial_rP_+\|_{L^\infty([0,T]\times {[r_1,r_2] })}\leq C_T,$$ for some $C_T>0$ which depends on $M_0$, $T$, $\lambda$, $\|B^0\|_{L^\infty}$, $\|E^0_\theta\|_{L^\infty}$, $\|B^b\|_{L^\infty}$, $\|E^b_\theta\|_{L^\infty}$, and $\|(1+p^0)f^0\|_{L^1}$. 

 Similarly, we can also obtain $\|\partial_rP_-\|_{L^\infty([0,T]\times {[r_1,r_2] })}\leq C_T,$ by using $$\partial_r =\frac{S-T_-}{1+\hat{p}_r},$$ where $T_-\eqdef \partial_t-\partial_r$ and $S\eqdef \partial_t +\hat{p}_r\partial_r$, in place of \eqref{splittingop}. This completes the proof.
\end{proof}
\subsection{Derivatives of \texorpdfstring{$f$}{}} Finally, we are ready to obtain an estimate for the derivatives of the solution $f$. 
\begin{lemma}Suppose $f\in C^2([0,T]\times {[r_1,r_2] }\times \rtw).$ Then we have $$\|f\|_{C^1([0,T]\times {[r_1,r_2] }\times \rtw)}\leq C_T,$$ for some constant $C_T>0$ which depends only on $M_0$, $T$, $\lambda$, $\|f^0\|_{L^\infty}$, the $C^1$ norm of $f^0$, $E_\theta^0$, $B^0$ on ${[r_1,r_2] }$, and the $C^1$ norm of $E^b_\theta$, $B^b$ on $[0,T]$.
\end{lemma}\begin{proof}
We start with differentiating the Vlasov equation in cylindrical-coordinates \eqref{pVlasov} with respect to $r$ variables. Then we observe that
\begin{multline*}
\left(\partial_t +\hat{p}_r\partial_r+\left(E_r+\hat{p}_\theta\bar{B}+\frac{p_\theta^2}{rp^0}\right)\partial_{p_r}+\left(E_\theta-\hat{p}_r\bar{B}-\frac{p_rp_\theta}{rp^0}\right)\partial_{p_\theta}\right)\partial_r f\\=-\partial_r\left(E_r+\hat{p}_\theta\bar{B}+\frac{p_\theta^2}{rp^0}\right)\partial_{p_r}f-\partial_r\left(E_\theta-\hat{p}_r\bar{B}-\frac{p_rp_\theta}{rp^0}\right)\partial_{p_\theta} f.
\end{multline*}Recall Lemma \ref{magconfi} and Lemma \ref{momsupp} and define \begin{equation}\label{Mbar}\bar{M}(t)\eqdef M_0+4\tilde{C}\frac{e^{Ct}}{C}.\end{equation}  Then we integrate this identity along the characteristics and obtain 
\begin{multline}\label{I1}\|\partial_r f(t)\|_{L^\infty((r_1,r_2)\times \{|p|\leq \bar{M}(t)\})}\leq \|\partial_r f^0\|_{{L^\infty([r_1+\delta_0,r_2-\delta_0]\times \{|p|\leq \bar{M}(t)\})}}\\
+\int_0^t  \left\|-\partial_r\left(E_r+\hat{p}_\theta\bar{B}+\frac{p_\theta^2}{rp^0}\right)\partial_{p_r}f-\partial_r\left(E_\theta-\hat{p}_r\bar{B}-\frac{p_rp_\theta}{rp^0}\right)\partial_{p_\theta} f\right\|_{L^\infty} \;ds\\
\lesssim  \|\partial_r f^0\|_{{L^\infty}}
+\int_0^t \left( \|E\|_{C^1}+\|B\|_{C^1}+\bar{M}(t)\right) \|\nabla_p f\|_{L^\infty} \;ds,
\end{multline}as $\left|\frac{p_r}{p^0}\right| \text{ and } \left|\frac{p_\theta}{p^0}\right|\leq 1.$
On the other hand, we differentiate \eqref{pVlasov} with respect to $p$ variables and obtain
\begin{multline*}
\partial_t\nabla_p f +\hat{p}_r\partial_r\nabla_pf
+\left(E_r+\hat{p}_\theta\bar{B}+\frac{p_\theta^2}{rp^0}\right)\partial_{p_r}\nabla_pf+\left(E_\theta-\hat{p}_r\bar{B}-\frac{p_rp_\theta}{rp^0}\right)\partial_{p_\theta}\nabla_pf\\=-\nabla_p(\hat{p}_r)\partial_r f-\nabla_p\left(E_r+\hat{p}_\theta\bar{B}+\frac{p_\theta^2}{rp^0}\right)\partial_{p_r}f-\nabla_p\left(E_\theta-\hat{p}_r\bar{B}-\frac{p_rp_\theta}{rp^0}\right)\partial_{p_\theta} f\\
=-\nabla_p(\hat{p}_r)\partial_r f-\nabla_p\left(\hat{p}_\theta\bar{B}+\frac{p_\theta^2}{rp^0}\right)\partial_{p_r}f-\nabla_p\left(-\hat{p}_r\bar{B}-\frac{p_rp_\theta}{rp^0}\right)\partial_{p_\theta} f.
\end{multline*}
We similarly integrate this identity along the characteristics and obtain 
\begin{multline}\label{I2}\|\nabla_p f(t)\|_{L^\infty((r_1,r_2)\times \{|p|\leq \bar{M}(t)\})}\leq \|\nabla_p f^0\|_{{L^\infty([r_1+\delta_0,r_2-\delta_0]\times \{|p|\leq \bar{M}(t)\})}}\\
+\int_0^t  \left\|-\nabla_p(\hat{p}_r)\partial_r f-\nabla_p\left(\hat{p}_\theta\bar{B}+\frac{p_\theta^2}{rp^0}\right)\partial_{p_r}f-\nabla_p\left(-\hat{p}_r\bar{B}-\frac{p_rp_\theta}{rp^0}\right)\partial_{p_\theta} f\right\|_{L^\infty} \;ds\\
\lesssim  \|\nabla_p f^0\|_{{L^\infty}}
+\int_0^t \left( \|\partial_r f\|_{L^\infty}+\left(\|B\|_{L^\infty}+|\bar{M}(t)|\right) \|\nabla_p f\|_{L^\infty} \right) \; ds
\end{multline}
By adding \eqref{I1} and \eqref{I2} and noting that the support of $f$ is in $(r_1,r_2)\times \{|p|\leq \bar{M}(t)\}$, we finally have
$$D(f)(t)\lesssim 1+\int_0^t D(f)(s) ds,$$ where $D(f)(s)\eqdef \|\partial_r f\|_{L^\infty}+\|\nabla_p f\|_{L^\infty}.$ This yields the bounds for $\|\partial_r f\|_{L^\infty}$ and $\|\nabla_p f\|_{L^\infty}$. Finally, we use the Vlasov equation $$\partial_t f=-\hat{p}_r\partial_rf-\left(E_r+\hat{p}_\theta\bar{B}+\frac{p_\theta^2}{rp^0}\right)\partial_{p_r}f-\left(E_\theta-\hat{p}_r\bar{B}-\frac{p_rp_\theta}{rp^0}\right)\partial_{p_\theta} f$$ and further obtain the bound for $\|\partial_t f \|_{L^\infty}.$ This completes the proof.
\end{proof}

\section{Global wellposedness: Proof of Theorem \ref{maintheorem1}}\label{sec:global}Together with all the estimates on the functions $f$, $E$, $B$, and their derivatives from the previous sections, we can obtain the global wellposedness of the problem \eqref{pVlasov}, \eqref{pMaxwell},\eqref{initial}, and \eqref{boundary}, and Hypothesis \ref{extmagassumption} as follows.

\noindent \textbf{(Existence)}
First of all, we obtain the existence of a global $C^1$ solution via the standard iteration argument which was introduced in the literature \cite{MR1066384,MR1463042,GL1996,MR1620506,MR1777635,MR3294216}, etc. We leave these standard details to the interested readers.

\noindent \textbf{(Uniqueness)}
Fix any time interval $[0,T]$ for some $T>0$. For the uniqueness, we suppose that there are two global $C^1$ solutions $(f^i, E^i, B^i)$ for $i=1,2$ to the problem \eqref{pVlasov}, \eqref{pMaxwell}, \eqref{initial}, and \eqref{boundary}, and Hypothesis \ref{extmagassumption}. Then we have
\begin{multline}\label{uniqueness}
\partial_t \tilde{f} +\hat{p}_r\partial_r\tilde{f}+\hat{p}_\theta \frac{1}{r}\partial_\theta \tilde{f}\\+\left(E^2_r+\hat{p}_\theta\bar{B}^{2}+\frac{p_\theta^2}{rp^0}\right)\partial_{p_r}\tilde{f}+\left(E^2_\theta-\hat{p}_r\bar{B}^2-\frac{p_rp_\theta}{rp^0}\right)\partial_{p_\theta} \tilde{f}\\
=-\left(\tilde{E}_r+\hat{p}_\theta\tilde{B}\right)\partial_{p_r}f^1-\left(\tilde{E}_\theta-\hat{p}_r\tilde{B}\right)\partial_{p_\theta} f^1,
\end{multline}where $\bar{B}^i=B^i+B_{ext}$ and $$\tilde{f}\eqdef f^1-f^2,\ \tilde{E}\eqdef E^1-E^2,\ \text{and}\ \tilde{B}\eqdef B^1-B^2.$$ Also, note that $\tilde{f}(0,r,p_r,p_\theta)=0$ for all $(r,p_r,p_\theta)\in (r_1,r_2)\times \rtw.$ By Section \ref{confinement}, we obtain that the characteristic trajectories never touch the boundaries. Thus, we integrate \eqref{uniqueness} along the characteristics and obtain that
\begin{multline}\label{ftilde}\|\tilde{f}(t)\|_{L^\infty((r_1,r_2)\times\rtw)}\leq \|\nabla_p f^1\|_{L^\infty([0,t]\times(r_1,r_2)\times\rtw)}\\\times \int_0^t \left(\|\tilde{E}(s)\|_{L^\infty((r_1,r_2))}+\|\tilde{B}(s)\|_{L^\infty((r_1,r_2))}\right)ds,
\end{multline}for each $t\in [0,T]$.
We first estimate the upperbound for $\|\tilde{E}_r(s)\|_{L^\infty((r_1,r_2))}$. Since $\partial_r(r \tilde{E_r}) = r\rho(\tilde{f})$ where $\rho(\tilde{f})$ is defined as $$\rho(\tilde{f})(t,r)=\int_{\mathbb{R}^2}\ \tilde{f}(t,r,p)dp,$$ we observe that 
$$\|\tilde{E}_r(s)\|_{L^\infty((r_1,r_2))}\leq\Blue{\frac{1}{r_1}} \int_{r_1}^{r_2} \int_{|p|\leq \bar{M}(s)}\  r\tilde{f}(s,r,p_r,p_\theta)\ dp_rdp_\theta dr\lesssim_T \|\tilde{f}(s)\|_{L^\infty},$$ where $\bar{M}$ is defined as \eqref{Mbar}.
For the estimates on $\|\tilde{E}_\theta(s)\|_{L^\infty((r_1,r_2))}$ and\newline $\|\tilde{B}(s)\|_{L^\infty((r_1,r_2))}$, we observe that $\tilde{E}_\theta$ and $\tilde{B}$ satisfy 
$$\partial_t (r\tilde{E}_\theta)+\partial_r(r\tilde{B})=\tilde{B}-rj_\theta(\tilde{f}),
$$
and $$
\partial_t(r\tilde{B})+\partial_r(r\tilde{E}_\theta)=0.
$$ Then since $\tilde{B}^0,\tilde{B}^b,\tilde{E}^0,$ and $\tilde{E}^b$ are all zero in \eqref{prop2.6}, we have 
\begin{multline}|\tilde{B}(t,r)|\leq
\frac{1}{r_1}\bigg(\int_0^t \|B(\tau)\|_{L^\infty({[r_1,r_2] })}d\tau
\\+\int_{\min\{t_1,t_2\}}^t\int_{|p|\leq \bar{M}(\tau)}r|\hat{p}_\theta||\tilde{f}(\tau,r-t+\tau,p_r,p_\theta)|dp_rdp_\theta d\tau\bigg),
\end{multline}by the definition  $j_\theta\eqdef \int_\rtw \hat{p}_\theta fdp_rdp_\theta$. Then by taking the supremum in $r$ variable and taking $|\hat{p}_\theta|\le 1$, we have
\begin{multline*}\|\tilde{B}(s)\|_{L^\infty((r_1,r_2))}\lesssim_T \sup_{r\in (r_1,r_2)}\int_0^s \int_{|p|\leq \bar{M}(\tau)}|\tilde{f}(\tau,r,p_r,p_\theta)|dp_rdp_\theta d\tau\\ \lesssim_T  \sup_{\tau\in [0,s]} \|\tilde{f}(\tau)\|_{L^\infty} .\end{multline*}Similarly we obtain  $\|\tilde{E}_\theta(s)\|_{L^\infty}\lesssim_T \sup_{\tau\in [0,s]} \|\tilde{f}(\tau)\|_{L^\infty}.$  Now we go back to \eqref{ftilde} and observe that
$$\tilde{u}(t)\lesssim_T \int_0^t \tilde{u}(s)ds,$$ where $\tilde{u}(s)\eqdef \sup_{\tau\in [0,s]} \|\tilde{f}(\tau)\|_{L^\infty},$ as we have $\|\nabla_p f^1\|_{L^\infty}\leq C_T$ for some $C_T>0.$ Since $\tilde{u}(0)=0,$ we obtain that $\tilde{u}(s)=0$ for any $s\in [0,T]$, and hence, $\tilde{f}(t)=\tilde{E}(t)=\tilde{B}(t)=0$ for any $t\in[0,T].$ This completes the proof for the uniqueness.  \qed

\bigskip

\noindent \textbf{(Non-negativity)} Suppose that $f^0$ is initially non-negative. Then $f$ is constant along the characteristics defined in Section \ref{confinement} and hence is non-negative.  \qed

%



\providecommand{\bysame}{\leavevmode\hbox to3em{\hrulefill}\thinspace}
\providecommand{\href}[2]{#2}

\end{document}